\documentclass[reqno,11pts]{amsart}
\usepackage{amssymb,amsmath,amsthm,amstext,amsfonts}
\usepackage{amsmath,amstext,amsthm,amsfonts}
\usepackage[dvips]{graphicx}
\usepackage[dvips]{graphicx}
\usepackage{color}

\makeatletter \@addtoreset{equation}{section} \makeatother

\renewcommand\thetable{\thesection.\@arabic\c@table}

\theoremstyle{plain}
\newtheorem{maintheorem}{Theorem}

\newtheorem{maincorollary}{Corollary}

\newtheorem{theorem}{Theorem }[section]

\newtheorem{lemma}[theorem]{Lemma}

\theoremstyle{definition} \theoremstyle{remark}
\newtheorem{remark}[theorem]{Remark}

\newtheorem{definition}[theorem]{Definition}

\newcommand{\field}[1]{\mathbb{#1}}

\newcommand{\torus}{\field{T}}

\newcommand{\al} {\alpha}

\newcommand{\de} {\delta}       

\newcommand{\vep}{\varepsilon}

\newcommand{\N}{\mathbb{N}}
\newcommand{\R}{\mathbb{R}}

\newcommand{\diam}{\operatorname{diam}}

\newcommand{\topp}{\operatorname{top}}

\renewcommand{\field}[1]{\mathbb{#1}}

\newcommand{\Ptop}{P_{\topp}}

\newcommand{\cP}{\mathcal{P}}
\newcommand{\cL}{\mathcal{L}}
\newcommand{\cE}{\mathcal{E}}
\newcommand{\cM}{\mathcal{M}}

\newcommand{\cA}{\mathcal{A}}

\def\ds{\displaystyle}

\begin{document}

\title{Multifractal analysis for weak Gibbs measures: from large deviations to irregular sets}

\author{ Thiago Bomfim and Paulo Varandas}

\address{Thiago Bomfim $\&$ Paulo Varandas, Departamento de Matem\'atica, 
Universidade Federal da Bahia\\
Av. Ademar de Barros s/n, 40170-110 Salvador, Brazil.}
\email{tbnunes@ufba.br}
\email{paulo.varandas@ufba.br }
\urladdr{http://www.pgmat.ufba.br/varandas}

\date{\today}

\begin{abstract}
In this article we prove estimates for the topological pressure of the set of points whose Birkhoff time averages
are far from the space averages corresponding to the unique equilibrium state that has a weak Gibbs property. In particular,
if $f$ has an expanding repeller and $\phi$ is an H\"older continuous potential we prove that the topological pressure of
the set of points whose accumulation values of Birkhoff averages belong to some interval $I\subset \mathbb R$
can be expressed in terms of the topological pressure of the whole system and the large deviations
rate function. As a byproduct we deduce that most irregular sets for maps with the specification property have topological pressure strictly smaller than the whole system.
Some extensions to a non-uniformly hyperbolic setting, level-2 irregular sets and hyperbolic flows are also given.

\end{abstract}


\maketitle

\section{Introduction}

Let $f : M \rightarrow M$ be a measurable transformation and $\mu$ an $f$-invariant and ergodic probability measure.
The celebrated Birkhoff's ergodic theorem asserts that for any given $\psi \in L^{1}(\mu)$
and for $\mu$-almost every $x \in M$
$$
\frac{1}{n} S_n\psi (x) :=\frac{1}{n}\sum_{i = 1}^{n - 1}\psi \circ f^{i}(x) \xrightarrow{} \int\psi \; d\mu
$$
as $n$ tends to infinity. On the other hand, despite the fact that from the ergodic point of view the set of
points where the Birkhoff averages do not converge is negligible it can be a
topologically large set or have full dimension. To illustrate this fact let us mention that if $f$ is continuous and
have the specification property then the set of points where the Birkhoff averages do not converge is either empty
or has total topological pressure with respect to any continuous potential (we refer the reader e.g.~\cite{Daniel} for details).
The study of the topological pressure or dimension of the these level sets multifractal can be traced back to
Besicovitch and this topic had contributions by many authors in the recent years (see \cite{BPS97, DK,  
Climenhaga, Feng, GR, JR,PW97, 
PW99, 
T,TV99, Daniel, Cli13, ZC13} and references therein).
Most commonly, given an observable $\psi$ and the decomposition
\begin{equation*}
M=\bigcup_{\alpha \in \mathbb R} M_\al  \; \cup \; E_\psi
\end{equation*}
where $M_\al=\{x \in M :  \lim_{n\to\infty} \frac{1}{n} S_n\psi (x) =\alpha \}$ and the irregular set $E_\psi$ is the
set of points for which the Birkhoff averages do not converge, one is interested in describing each of the previous sets
from the topological, dimensional or ergodic point of view.
Motivated by the aforementioned results by Thompson~\cite{Daniel} and the special ergodic theorem proved by  Kleptsyn, Ryzhov and Minkov~\cite{special} that prove that the Haussdorf dimension of the deviation set for SRB measure is smaller than the dimension of the manifold provided a large deviations property, one of our aims 
in this article is to provide a multifractal description of more general sets. 

Let us consider the sets
$$
\overline{X}_{\mu,\psi , c}
    = \Big\{ x \in M :
    \limsup_{n\to \infty} \Big|\frac{1}{n}\sum_{j=0}^{n-1}\psi(f^j(x)) - \int \psi\, d\mu \Big|
    \geq c \Big\}
$$
and
$$
\underline{X}_{\mu,\psi , c}
    = \Big\{ x \in M :
    \liminf_{n\to \infty} \Big|\frac{1}{n}\sum_{j=0}^{n-1}\psi(f^j(x)) - \int \psi\, d\mu \Big|
    \geq c \Big\}
$$
where $\mu$ is an $f$-invariant and ergodic probability measure, $\psi$ is an observable and $c > 0$
and to study them from the topological pressure viewpoint (more general sets will be defined later on).
Clearly we have the inclusion
$\underline{X}_{\mu,\psi , c}\subset \overline{X}_{\mu,\psi , c}$ for all $c>0$.  In many cases we are interested
in studying $f\mid_\Lambda$ we will consider the corresponding sets $\underline{X}_{\mu,\psi , c}\cap \Lambda$
and $\overline{X}_{\mu,\psi , c}\cap \Lambda$.
When no confusion is possible, for notational simplicity
we shall omit the dependence of the sets on $\psi$, $\mu$ and $\Lambda$, and write
simply $\underline{X}_{c}$ and $\overline{X}_{c}$.
If $J\subset \mathbb R$ is an interval we define $X(J)$ as the set of points $x$ so that
the following limit exists and $\lim\frac1n S_n\psi (x)\in J$ and let $X(c)$ denote the case when
one considers the degenerate interval $J=[c,c]$. One motivation is to consider the decomposition
of the set of points whose time averages do not converge to the
space average by
$$
M\setminus
\Big\{x \colon \frac{1}{n}\sum_{i = 1}^{n - 1}\psi \circ f^{i}(x) \xrightarrow{} \int\psi \; d\mu\Big\} =\bigcup_{c> 0} \overline{X}_c = \bigcup_{c> 0} \underline{X}_c
$$
and to study the continuity, monotonicity and concavity of the pressure functions
$$
c\mapsto P_{\overline{X}_c} (f,\phi)
    \quad\text{and}\quad
c\mapsto P_{\underline{X}_c} (f,\phi)
$$

Our first main purpose here is to study the previous functions in a context of dynamical systems admitting 
equilibrium states that exhibit a weak Gibbs property. Roughly, we prove that the previous
topological pressure functions are bounded from above by the topological pressure of the dynamical system
with an error term given by an exponential large deviations rate (see Theorem~\ref{thm:exp-add} and
Corollaries~\ref{thm:exp-add1} and ~\ref{cor:thompson} for precise statements). 
Furthermore, in a context of uniform hyperbolicity, we prove that $P_{\overline{X}_c} (f,\phi) = P_{\underline{X}_c} (f,\phi)$ and provide upper and lower bounds which allow us to conclude that the pressure function 
$c\mapsto P_{\overline{X}_c} (f,\phi)$ is differentiable, concave and
strictly decreasing in $c$ and varies continuously along continuous parametrized families of dynamical
systems (see Theorem~\ref{thm:precise.regularity} for the precise statement).
Similar results for the multifractal analysis of level-2 sets, meaning the analysis of level and irregular sets for Birkhoff averages in the space of probability measures, are also obtained (see Theorem~\ref{thm:level-2}).  
Hence, the connection between large deviations and multifractal analysis revealed to be very fruitful.

Our second main purpose was to provide a finer description of the irregular set $E_\psi$. 
Although a dynamical system which admits some hyperbolicity and a unique equilibrium state (that has
some weak Gibbs property) with exponential large deviations estimates can be proved to verify that 
the topological pressure of the sets $\underline{X}_{\mu,\psi , c}$ and $\overline{X}_{\mu,\psi , c}$ is strictly smaller than the topological pressure of the whole system, one cannot expect immediate estimates for the irregular set. In fact, the irregular set $E_\psi$
may not be contained in neither of the sets above for some fixed $c$.  This motivates the decompositions
$E_\psi=\cup_{c>0} \overline{E}_{\mu,\psi,c}$ and $E_\psi=\cup_{c>0} \underline{E}_{\mu,\psi,c}$ with
\begin{align*}
\overline{E}_{\mu,\psi,c} = E_\psi \cap \overline{X}_{\mu,\psi , c}
    \quad\text{and}\quad
\underline{E}_{\mu,\psi,c} = E_\psi \cap \underline{X}_{\mu,\psi , c}.
\end{align*}
In other words, the set $\underline{E}_{\mu,\psi,c}$ consists of points $x\in M$ whose Birkhoff averages
$\frac1n \sum_{j=0}^{n-1} \psi(f^j(x))$ not only do not converge as they remain at distance larger than $c$
from the time average $\int \psi \,d\mu$ for all large $n$. Finally, the set $\overline{E}_{\mu,\psi,c}$ consists of
points $x\in M$ whose Birkhoff averages  $\frac1n \sum_{j=0}^{n-1} \psi(f^j(x))$ do not converge and
have infinitely many values of $n$ so that the Birkhoff averages  $\frac1n \sum_{j=0}^{n-1} \psi(f^j(x))$
remain at distance larger than $c$.  
As a direct consequence of our results, even for maps with specification property, irregular sets 
$\underline{E}_{\mu,\psi,c}$ and $\overline{E}_{\mu,\psi,c}$ have topological pressure strictly smaller than the topological pressure of the whole system. In fact, we can indeed prove some lower bounds for the 
topological pressure of these irregular sets and, consequently, to study its regularity as a function of
the parameter $c$ (c.f. Corollary~\ref{cor:thompson2}).
An extension to level-2 irregular sets is given in Section~\ref{Statement of the main results} while applications for hyperbolic maps and flows and non-uniformly hyperbolic dynamics are given in Section~\ref{Examples}. In particular, an example from \cite{DGR11, DG12} is given
that illustrates the case where the pressure function is discontinuous and not strictly decreasing.
An extension of the current results to the study of irregular sets for non-additive sequences of observables 
was carried out in \cite{BV14}, while we expect that these results can be also extended to the class 
of partially hyperbolic diffeomorphisms in \cite{CN14}. 

This article is organized as follows. In Section~\ref{Statement of the main results} we provide some definitions
and state the main results. Section~\ref{Proofs} is devoted to the proof of the results. 
A large amount of applications and examples is given in Section~\ref{Examples}.

\section{Statement of the main results}\label{Statement of the main results}

In this section we introduce some necessary notions and state the main results.

\subsection{(Weak) Gibbs measures}

In many cases equilibrium states arise as invariant measures absolutely continuous with respect to probability
measures exhibiting some Gibbs property. Let us now describe this wide class of measures. Given $\vep>0$,
$n\ge 1$ and $x\in M$ the $(n,\vep)$-\emph{dynamical ball} $B(x,n,\vep)$ is the set of points $y\in M$ so that
$d(f^j(x),f^j(y))<\vep$ for all $0\le j \le n-1$.

\begin{definition}
Given a potential $\phi$ we say that a probability $\nu$ is a \emph{weak Gibbs measure} with respect the $\phi$ on
$\Lambda \subset M$ if there exists $\vep_0>0$ so that for every $0<\vep<\vep_0$ there exists $K(\vep)>0$, for $\nu$-almost every $x$ there exists a sequence $n_k(x)\to \infty$ such that
$$
K(\vep)^{-1} \leq \frac{\nu(B(x,n_{k}(x),\vep))}{e^{-n_{k}(x)P + S_{n_{k}(x)}\phi(x)}} \leq K(\vep),
$$
where $S_n\phi=\sum_{j=0}^{n-1} \phi \circ f^j$ denotes the usual Birkhoff sum.
If the later condition holds for all positive integers $n$ (independently of $x$) we will say that
$\nu$ is a \emph{Gibbs measure} with respect the $\phi$.
\end{definition}

In the later notion of weak Gibbs one does not require the sequence of times to have positive density at infinity
in the set of integers.
Naturally, in applications it is most interesting case is when the value $P$ in the previous expression coincides
with the topological pressure $\Ptop(f,\phi)$. Such measures arise naturally in the context of (non-uniform) hyperbolic dynamics. Given a basic set $\Omega$ for a diffeomorphism $f$ Axiom A (or $\Omega$ 
repeller to $f$) it is known that every potential $\phi$ satisfying
\begin{equation}\label{eq:bowen}
\exists A, \delta > 0 : \sup_{n \in \N}\gamma_{n}(\phi,\delta) \leq A,
\end{equation}
where $\gamma_{n}(\phi,\delta) := \sup\{|S_n\phi(y) - S_n\phi(z)| : y,z \in B(x,n,\delta) \},$ admits a unique equilibrium state $\mu_ {\phi}$ and it is a Gibbs measure. This condition, introduced by Bowen \cite{Bowen} to prove uniqueness of equilibrium states for expansive maps with the specification property
it is known as \emph{Bowen condition}.

\subsection{Statement of the main results}

This section is devoted to the statement of the main results. For that purpose we shall introduce some
definitions and notations.  The first result provides a topological counterpart to the special ergodic theorem.
In that follows, given a continuous function $\psi : M \rightarrow \mathbb R$, a probability measure $\mu$
and a closed set $I \subset \mathbb R$ 
we denote
$$
\overline{X}_{I}
    = \Big\{ x \in M :
    \limsup_{n\to \infty} \frac{1}{n}\sum_{j=0}^{n-1}\psi(f^j(x)) \in I\Big\}
$$
and analogously
$$
\underline{X}_{I}
    = \Big\{ x \in M :
    \liminf_{n\to \infty} \frac{1}{n}\sum_{j=0}^{n-1}\psi(f^j(x)) \in I\Big\}.
$$
Moreover, given $\delta > 0$ we denote by $I_{\delta}$ the $\delta$-neighborhood of the set $I$.
Finally, given a probability measure $\nu$ let us define the large deviations upper bound
\begin{equation}\label{eq:Lc}
L_{I,\nu} := -\limsup_{n \to +\infty}\frac{1}{n}\log\nu \Big(\{x \in \Lambda : \frac{1}{n}S_{n}\psi(x) \in I \}\Big).
\end{equation}
We are now in a position to state our first main result.

\begin{maintheorem}\label{thm:exp-add}
Let $M$ be a compact metric space, $f: M\to M$ be a continuous map, $\phi:M \to \mathbb R$ be a
continuous potential, $\nu$ be a (not necessarily invariant) Gibbs measure on $M$
and $\mu_{\phi} \ll \nu$ be the unique equilibrium state of $f$ with respect the $\phi$.
Then, for any continuous $\psi: M\to \mathbb R$, any closed interval $I \subset \mathbb \R$ and
any small $\delta$,
$$
P_{\underline{X}_{I} }(f , \phi) \leq P_{\overline{X}_{I} }(f , \phi)
	    \leq P_{\topp}(f , \phi) - L_{I_{\delta},\nu}\;
    		\le \;\Ptop (f,\phi).
$$
\end{maintheorem}

In fact, it follows from \cite[Theorem~2.1]{Va12} that, since $\nu$ is a (strong) Gibbs measure, if
$\int \psi d\mu_{\phi} \notin I_{\delta}$ then the large deviations  property that $L_{I_{\delta}}>0$
holds and, consequently, the topological pressure of  the sets $\underline{X}_{I}$ and $\overline{X}_{I}$
is strictly smaller than  $\Ptop(f,\phi)$. When no confusion is possible we shall omit the dependence
of $L_{I,\nu}$ on $\nu$. Our result is applicable to the case of topological repellers.

\begin{remark}\label{rmk:hyperbolic}
We notice that Theorem~\ref{thm:exp-add} also holds for bilateral subshifts of finite type
and \emph{locally H\"older} observables. In fact, given such an observable $g$ there exists
$\psi$ that is constant along local stable leaves (depends only on future coordinates of the shift) and such that
$g= \psi + u\circ f - u$ for some continuous $u$ (see~\cite{Bow75}). Thus
$
\left| \frac1n S_n g (x) -  \frac1n S_n \psi (x) \right|
	\leq \frac{2\| u \|_0}{n}
$
tends to zero (uniformly) as approaches infinity and, consequently, $X_I(g)=X_I(\psi)$, $\overline{X}_I(g)=\overline{X}_I(\psi)$
and $\underline{X}_I(g)=\underline{X}_I(\psi)$ for all intervals $I\subset \mathbb R$. Hence, using the Gibbs property
and replacing dynamic balls by cylinders associated to the Markov partition same conclusions of 
Theorem~\ref{thm:exp-add} still hold.
\end{remark}

\begin{definition}
Given a compact metric space $(M,d)$ and a continuous open map $f: M\to M$ we say that an $f$-invariant set
$\Lambda \subset M$ is a \emph{repeller} for $f$ if  there are $C,\lambda,\vep>0$ so that 
$d(f^n(x),f^n(y)) \ge C e^{\lambda n} d(x,y)$ for all $y\in B(x,n,\vep)$ and $n\ge 1$.
\end{definition}

It is clear that the later holds for smooth expanding maps. 
Recall that an observable $\psi: M\to \mathbb R$ is \emph{cohomologous to a constant} if there exists a
constant $c$ and a measurable function $u$ so that $\psi=u\circ f - u + c$.

\begin{definition}
 Given an observable $\psi: M\to \mathbb R$ and $t\in \R$ the  \emph{free energy} $\cE_{f,\phi,\psi} $ is
\begin{equation*}
\cE_{f,\phi,\psi}(t)
    = \limsup_{n\to\infty} \frac1n \log \int e^{t S_n  \psi} \,d\mu_{f,\phi}.
\end{equation*}
\end{definition}

In many cases, e.g. when the transfer operator associated to the potential $\phi$ has a spectral gap property, the expression in the right hand side does converge to
$$
\cE_{f,\phi, \psi}(t)
    :=\lim_{n\to\infty} \frac1n \log \int e^{t S_n \psi} \; d\mu_{f,\phi}
    = \Ptop(f, \phi +t\psi) -\Ptop(f, \phi).
$$
If this is the case and the topological pressure is smooth
then $t \mapsto \cE_{f,\phi,\psi}(t)$ is affine if $\psi$ is cohomologous to a constant and otherwise
$t \mapsto \cE_{f,\phi,\psi}(t)$ is strictly convex in some interval $J=[t_-,t_+]$ and one can associate the
\emph{``local"   Legendre transform} $I_{f,\phi,\psi}$ given by
\begin{equation*}
I_{f,\phi,\psi}(s)
    = \sup_{t\in J} \; \big\{ s\, t-\cE_{f,\phi,\psi}(t) \big\}
\end{equation*}
and well defined in the interval $ [\cE_{f,\phi,\psi}'(t_-), \cE_{f,\phi,\psi}'(t_+)]$.
The interval $J$ may depend on $f$, $\phi$ and $\psi$ and that $I_{f,\phi,\psi}(s)$ can often
be proved to be a (local) level-1 large deviations \emph{rate function} (see e.g.~\cite{Young, RY08, Bomfim}): for all $[a,b]\subset J$
\begin{equation}\label{eq:lld1}
\limsup_{n\to\infty} \frac1n \log \nu_{f,\phi}
    \left(x\in M : \frac1n S_n\psi(x) \in [a,b] \right)
    \le-\inf_{s\in[a,b]} I_{f,\phi,\psi}(s)
\end{equation}
and
\begin{equation}\label{eq:lld2}
\liminf_{n\to\infty} \frac1n \log \nu_{f,\phi}
    \left(x\in M : \frac1n S_n\psi(x) \in (a,b) \right)
    \ge-\inf_{s\in(a,b)} I_{f,\phi,\psi}(s)
\end{equation}
As a byproduct of our previous result and the large deviations property 
and the fact that  equilibrium states associated to H\"older continuous potentials satisfy the Gibbs property we deduce the
following:

\begin{maincorollary}\label{thm:exp-add1}
Let $f: M\to M$ be a continuous map, $\Lambda \subset M$ be a transitive repeller, $\phi:M \to \mathbb R$ be
an H\"older continuous potential and $\mu=\mu_{f,\phi}$ be the unique equilibrium state for $f\mid_\Lambda$ with respect to $\phi$. Then, for any continuous observable $\psi: M\to \mathbb R$ and $c>0$
$$
P_{\underline{X}_{c} }(f , \phi) \leq P_{\overline{X}_{c} }(f , \phi)
    \leq P_{\topp}(f , \phi) - L_{c -\delta}
    <\Ptop (f,\phi)
$$
for every small $\delta$, where $L_{c} := L_{I_c}$ is defined as in \eqref{eq:Lc} with respect to
$I_c=(-\infty, \int \psi d\mu_{\phi} -c] \cup [\int \psi d\mu_{\phi} + c, +\infty)$.
\end{maincorollary}

Since, in the previous results, the topological pressure  is strictly smaller than the topological
pressure $\Ptop(f,\phi)$, this has particularly interesting applications in connection with the specification
property. Recall that a system satisfies the \emph{specification property} if
for any $\vep>0$ there exists an integer $N=N(\vep)\geq 1$ such that
the following holds: for every $k\geq 1$, any points $x_1,\dots,
x_k$, and any sequence of positive integers $n_1, \dots, n_k$ and
$p_1, \dots, p_k$ with $p_i \geq N(\vep)$ 
there exists a point $x$ in $M$ such that
$$
\begin{array}{cc}
d\Big(f^j(x),f^j(x_1)\Big) \leq \vep, &\forall \,0\leq j \leq n_1
\end{array}
$$
and
$$
\begin{array}{cc}
d\Big(f^{j+n_1+p_1+\dots +n_{i-1}+p_{i-1}}(x) \;,\; f^j(x_i)\Big)
        \leq \vep &
\end{array}
$$
for every $2\leq i\leq k$ and $0\leq j\leq n_i$.
We also obtain the following result:

\begin{maincorollary}\label{cor:thompson}
Let $f: M\to M$ be a continuous map admitting a transitive
repeller $\Lambda \subset M$ and $\psi: M\to \mathbb R$
be such that the set of of irregular points satisfies $E_\psi \neq \emptyset$. Then 
$
\Ptop(f,\phi) = P_{E_\psi}(f,\phi) > P_{\overline{X}_c}(f,\phi)
$
for every $c>0$.
\end{maincorollary}

In fact,  it follows from~\cite{Daniel} that a dynamical system with the specification property is such that
irregular sets are either empty or have full topological pressure with respect to any continuous potential.
Since the dynamical systems restricted to the transitive repeller satisfies the specification property then the
first equality follows from \cite{Daniel}.  In particular, using
$$
E_\psi=\bigcup_{n\ge 1} \, [E_\psi\cap \overline{X}_{1/n}]
$$
and also $P_{E_\psi}(f,\phi)= \sup_{n\ge 1} P_{E_\psi\cap \overline{X}_{1/n}}(f,\phi)$ the previous corollary roughly
means that  despite the set of irregular points having full topological pressure, the ones that give a larger contribution
to the topological pressure are those with time averages which are infinitely often very close to the mean.

One could wonder if there could exist a strict inequality
$P_{\underline{X}_{c} }(f , \phi) < P_{\overline{X}_{c} }(f , \phi)$ and what is the regularity of the topological
pressure of those subsets. The next theorem provides an answer to these questions under the assumption of
uniform expansion.

\begin{maintheorem}\label{thm:precise.regularity}
Let $f: M\to M$ be a continuous map admitting a mixing repeller $\Lambda \subset M$,
$\phi:M \to \mathbb R$ be a continuous potential so that $\mu_\phi$ is the unique equilibrium state for
$f$  with respect to $\phi$ and $\mu_\phi \ll \nu$ where $\nu$ is a Gibbs measure. If $\phi,\psi$ satisfy the Bowen condition, $\psi$ is not cohomologous to a constant and $\int\psi\,d\mu_{f,\phi}=0$ then
$$
P_{\overline{X}_{c} }(f , \phi) \leq P_{\topp}(f , \phi) - \min\{I_{f,\phi,\psi}(- c) \;,\; I_{f,\phi,\psi}(c)\}
$$
where $I_{f,\phi,\psi}$ is the large deviations rate function.
If $0 \notin [c_1,c_2]$ and $c=\min\{|c_1|,|c_2|\}$ then
either $\overline{X}_{c} =\emptyset$ or
\begin{align*}
P_{\overline{X}_{c}}(f,\phi)
    & = P_{\underline{X}_{c}} (f,\phi) = P_{X(c_*)} (f,\phi) 
    = P_{\topp}(f , \phi) - I_{f,\phi,\psi}(c_*).
\end{align*}
where
\begin{equation}\label{eqces}
c_*=
	\begin{cases}
	c, \text{ if } I_{f,\phi,\psi}(c) < I_{f,\phi,\psi}(-c) \\
	-c, \qquad \text{ otherwise. }
	\end{cases}
\end{equation}
In particular $\mathbb R^+_0 \ni c\mapsto P_{\overline{X}_c}(f,\phi)$ is differentiable,
concave and strictly decreasing.
Furthermore, the right hand side expression varies continuously with $c$ and also varies continuously
with $\phi$, $\psi$ in the $C^\alpha$-topology.
Moreover, if $V$ is a compact metric space and
$V \ni v \mapsto (f_{v})_v$  is a continuous (in the $C^1$-topology) family of expanding maps on $M$
then $v\mapsto P_{\overline{X}_c}(f_v,\phi)$ is also a continuous function.
\end{maintheorem}

\begin{figure}[h]
\includegraphics[scale=.45]{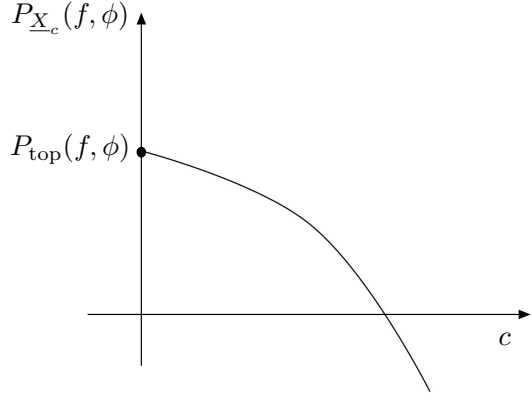}
\vspace{-.5cm}
\caption{Continuity, monotonicity and concavity of the pressure function }
\label{Figure1}
\end{figure}

Under the previous assumptions we can provide a more detailed description of the irregular sets
$\overline{E}_c$ as follows.

\begin{maincorollary}\label{cor:thompson2}
Let $f: M\to M$ be a continuous map admitting a mixing repeller $\Lambda \subset M$,
$\phi:M \to \mathbb R$ be a continuous potential so that $\mu_\phi$ is the unique equilibrium state for
$f$  with respect to $\phi$ and $\mu_\phi \ll \nu$ where $\nu$ is a Gibbs measure. Set $\overline{E}_c=\overline{X}_c \cap E_\psi$ the irregular set  contained in $\overline{X_c}$. If $E_\psi\neq \emptyset$ then 
for every $c>0$:
\begin{enumerate}
\item $\Ptop(f,\phi) = P_{E_\psi}(f,\phi) > P_{\overline{X}_c}(f,\phi) \ge P_{\overline{E}_c}(f,\phi)$,
\item if $\overline{E}_c \neq \emptyset$ then $P_{\overline{X}_c}(f,\phi) = P_{\overline{E}_c}(f,\phi)$ and 
	 $c\mapsto P_{\overline{E}_c}(f,\phi)$ is differentiable, concave and strictly decreasing.
\end{enumerate}
\end{maincorollary}

Actually, in this setting we can provide also estimates for irregular sets corresponding to empirical measures
$\delta_{x,n}:=\frac1n\sum_{j=0}^{n-1} \delta_{f^j(x)}$. Let $\mathcal M_1$ denote the set of probability measures
on $M$ and let $d$ be any metric compatible with the weak$^*$ topology
(e.g. $d(\mu,\nu)=\sum_{k\ge 1} \frac1{2^k \|g_k\|_0} |\int g_k \,d\mu-\int g_k \,d\nu|$ for some countable and dense subset $(g_k)_k$ of continuous observables). We say a \emph{level-2 large deviations principle} holds for $\nu$  if there is a lower semicontinuous
function $Q: \mathcal M_1 \to [0,+\infty]$ so that
$$
\limsup_{n\to\infty} \frac1n \log \nu_{f,\phi}
    \left(x\in M : \delta_{x,n} \in U \right)
    \le-\inf_{\eta \in U} Q(\eta)
$$
for every closed set $U\subset \mathcal M_1$ and
$$
\liminf_{n\to\infty} \frac1n \log \nu_{f,\phi}
    \left(x\in M : \delta_{x,n} \in V \right)
    \ge -\inf_{\eta \in V} Q(\eta)
$$
for every open set $V\subset \mathcal M_1$.
Level-2 large deviations principles in dynamical systems
have been obtained e.g. in \cite{CRL98, CTY13}. 
Consider
$$
\overline{Y}_{\mu, c}
    \!=\! \{ x \in M :
    \limsup_{n\to \infty} d (\delta_{x,n},\mu) \geq c \}
\quad\text{and}\quad
\underline{Y}_{\mu, c}
    \!=\! \{ x \in M :
    \liminf_{n\to \infty} d (\delta_{x,n},\mu) \geq c \}
$$
and, for $C \subset \mathcal{M}_{1}$, define $Y(C) := \{x \in M : \lim_{n \to +\infty}\delta_{x,n} \in C \}$.

Part of the strategy can be used to estimate the topological pressure of points with specified behaviour of the
empirical measures for dynamical systems that have the $g$-almost product structure and uniform separation property.
These notions, introduced by C. Pfister and W. Sullivan \cite{PS05}, are strictly
weaker than the specification property and the positive expansive property, respectively.
In fact, the uniform separation property is true even for asymptotically entropy-expansive maps.
Let us  recall these notions.

\begin{definition}
Let $M$ be a compact metric space and $f : M \rightarrow M$ be continuous.
A nondecreasing unbounded map $g : \N \rightarrow \N$ is a \emph{blow-up function} if
$g(n) < n$ for all $n$ and $\lim_{n \to +\infty} g(n) / n = 0$.
\end{definition}

For any subset of integers $\Lambda\subset [0,N]$, we will use the
family of distances in the metric space $X$ given by
$d_\Lambda(x,y)=\max \{d(f^ix,f^iy):i\in\Lambda\}$ and consider
the balls $B_{\Lambda}(x,\vep)=\{y\in X:d_\Lambda(x,y)<\vep\}$.
Given a blow-up function $g$, $\vep>0$ and $n\ge 1$, the \emph{$g$-mistake dynamical ball}
$B_n(g;x,\vep)$ of radius $\vep$ and length $n$ associated to $g$ is defined by
\begin{eqnarray*}
B_n(g;x,\vep)=\{ y\in X\mid y\in B_{\Lambda}(x,\vep)~\hbox{for
some}~\Lambda\in I(g;n,\vep)\}
=\bigcup_{\Lambda\in I(g;n,\vep)}B_{\Lambda}(x,\vep)
\end{eqnarray*}
where $I(g;n,\vep)=\{ \Lambda\subset [0,n-1]\cap\mathbb{N}\mid
\# \Lambda \geq n-g(n)\}$. We are now in the position to define the
$g$-almost product property.

\begin{definition}
Let $g$ be a blow-up function.  The continuous map $f : M \rightarrow M$ has the
\emph{$g$-almost product property} if there exists a nonincreasing function $m : \mathbb{R}^{+} \rightarrow \N$, such
that for any $k \in \N$, any points $x_{1}, x_2,  \ldots, x_{k}$, any positive $\vep_{1}, \ldots \vep_{k}$ and any integers $n_{i} \geq m(\vep_{1})$ for $i=1\dots k$ it holds that
$
\bigcap_{j=1}^{k}f^{-M_{j-1}}B_{n_{j}}(g ; x_{j} , \vep_{j}) \neq \emptyset.
$
where $M_{0}=0$ and $M_{i}=n_{1}+n_{2}+\cdots+n_{i},i=1,2,\cdots,k-1.$
\end{definition}

Given $\delta,\vep>0$ and $n\ge 1$ we say that two points $x,y\in X$ are \emph{$(\delta,n,\vep)$-separated}
if $\#\{ 0 \le j \le n-1 : d(f^j(x),f^j(y))>\vep \} \ge \delta n$. In addition, a set $E\subset X$ is $(\delta,n,\vep)$-separated
if all pairs of distinct points in $E$ are $(\delta,n,\vep)$-separated. This means that the moments at which the
two pieces of orbit are $\vep$-separated form a $\delta$-proportion.

\begin{definition}
A continuous map $f : M \rightarrow M$ has the \emph{uniform separation property} if
for any $\eta$ there exists $\delta^{} > 0$ and $\vep^{} > 0$ so that for any ergodic probability measure
$\mu$ and any neighborhood $F$ of $\mu$ in the space of all probability measures $\mathcal M_1$
there exists $n^{}_{F,\mu,\eta}\ge 1$ such that
$$
N(F; \delta^{},n, \vep^{})
    \geq \exp \; [ n(h_\mu(f) - \eta)]
$$
for all $n \geq n^{}_{F,\mu,\eta}$, where $N(F; \delta^{},n, \vep^{})$ is the maximal cardinality of a
$(\delta^{},n, \vep)$-separated  subset of the set $\{x \in M : \delta_{x,n} \in F\}.$
\end{definition}

Taking these notions in account we
also obtained the following result.

\begin{maintheorem}\label{thm:level-2}
Let $f: M\to M$ and $\phi : M \rightarrow \R$ be continuous, $\nu$ be a (not necessarily invariant) Gibbs measure and
assume $\mu=\mu_{f,\phi}\ll \nu$ is the unique equilibrium state for $f$ with respect to $\phi$. Assume the metric $d$ on
$\mathcal{M}_{1}$ has the following properties:
\begin{itemize}
 \item[i.] $d(\eta_{1} + \eta , \eta_{2} + \eta) = d(\eta_{1}, \eta_{2}), \forall \eta_{1},\eta_{2},\eta \in \mathcal{M}_{1}; $
 \item[ii.] $d(t\eta_{1} , t\eta_{2}) = td(\eta_{1} , \eta_{2}), \forall \eta_{1}, \eta_{2} \in \mathcal{M}_{1} \;\;\mbox{and}\;\; t >0,$
\end{itemize}
If a level-2 large deviations principle holds for $\nu$ then for every $c>0$
$$
P_{\overline{Y}_{\mu,c} }(f , \phi) \leq P_{\topp}(f , \phi) - \inf_{d(\eta ,\mu) \geq c}Q(\eta)
		\leq \Ptop(f,\phi).
$$
In addition, if $f$ satisfies the almost product and uniform separation properties and $ 0< c_1 < c_{2}$ then
either $\overline{Y}_{\mu,c_1} =\emptyset$ or
\begin{align*}
P_{\overline{Y}_{\mu,c_1}}(f,\phi)
    & = P_{\underline{Y}_{\mu,c_1}} (f,\phi) = P_{Y(\partial B(\mu,c_{1}))} (f,\phi) = P_{Y(B(\mu,c_{1}))} (f,\phi)  \\
    & = P_{Y\big(\overline{B(\mu,c_{1}, c_{2})}\big)}(f,\phi) = P_{Y(B(\mu,c_{1},c_{2}))}(f,\phi)= P_{\topp}(f , \phi) - \inf_{d(\eta ,\mu) = c_{1}}Q(\eta),
\end{align*}
where $B(\mu,c_{1})$ denotes the ball of radius $c_1$ around $\mu$ and $B(\mu,c_{1}, c_2)$ denotes the
annulus $\{\eta \in \cM(X): c_1<d(\eta,\mu)<c_2\}$.
\end{maintheorem}

Further information can be extracted if one knows the behaviour of the rate function $Q$, in which case one
can prove the topological pressure of the level sets is strictly smaller than the topological pressure
$P_{\topp}(f , \phi) $. This is the case for repellers as we now detail.

\begin{maincorollary}\label{thm:level-21}
Let $f: M\to M$ be a continuous map admitting a transitive repeller $\Lambda \subset M$, $\phi:M \to \mathbb R$ is a
continuous potential and there exists a unique equilibrium state $\mu_\phi$ for $f$ with respect to $\phi$
and it is a Gibbs measure under $\Lambda$. Then, for all $ 0< c_1 < c_{2}$
either $\overline{Y}_{\mu,c} =\emptyset$ or
\begin{align*}
P_{\overline{Y}_{\mu,c_1}}(f,\phi)
    & = P_{\underline{Y}_{\mu,c_1}} (f,\phi) 
    	= P_{\topp}(f , \phi) - \inf_{d(\eta ,\mu) = c_{1}}Q(\eta)
	< \Ptop(f,\phi)
\end{align*}
where $Q(\eta)=P_{\text{top}}(f,\phi)-h_\eta(f)+\int \psi \,d\eta$.
\end{maincorollary}

The previous result implies that the set of irregular points whose range of values of Birkhoff averages
are far from the corresponding value associated to the equilibrium state have topological pressure smaller
than $\Ptop(f,\phi)$. In particular, this shows that in order to build an irregular set of points with
large topological pressure one needs to use some specification property and points whose empirical measures
are arbitrarily close to the equilibrium state. In some sense this means the classical construction of irregular sets
with large topological pressure is optimal.
Our next results apply for weak Gibbs measures.

\begin{maintheorem}\label{thm:nue-add}
Let $M$ be a compact metric space, $f: M\to M$ be a continuous
map, $\phi:M \to \mathbb R$ be a continuous potential, $\nu$ be a (not necessarily invariant) weak Gibbs measure
and $\mu_{\phi}\ll \nu$ be the unique equilibrium state of $f$ with respect to $\phi$. For any continuous
$\psi: M\to \mathbb R$ and closed interval $I \subset \mathbb \R$ it holds that
$$
P_{\underline{X}_{I} }(f , \phi) \leq P_{\topp}(f , \phi) - L_{I_{\delta}}.
$$
for every small $\delta$. If, in addition, $L_{I_{\delta}}<0$ then $P_{\underline{X}_{I} }(f , \phi) <\Ptop (f,\phi)$.
\end{maintheorem}

Estimates for $L_{I_\de}$ will depend on the weak Gibbs property and some can
be found in \cite{Va12}. Actually we can indeed prove a version of the previous
results in the non-uniformly expanding setting. Given $\sigma,\delta>0$ we define $H =H(\sigma,\delta)$ as the set of points in $\Lambda$ with infinitely
many $(\sigma,\delta)$-hyperbolic times (see e.g. \cite{Varandas2} for a precise definition). We will say that
an $f$-invariant probability measure $\mu$ is \emph{expanding} if  $\mu(H(\sigma,\delta))=1$ for some
positive constants $\sigma,\delta$.
Moreover, given $\psi$ continuous, we will say that we have an \emph{exponential large deviations upper bound} if
\begin{equation}\label{eq:upperbound}
\limsup_{n \to +\infty}\frac{1}{n}\log\mu_{\Phi}\Big(\{x \in M : \Big|\frac{1}{n}S_n\psi(x) - \int \psi \,d\mu_\phi \Big| \geq c\}\Big) < 0
\end{equation}
for all $c > 0$. A direct consequence of the previous abstract result is as follows.

\begin{maincorollary}\label{thm:nue-add1}
Let $f: M\to M$ be a $C^{1+\al}$-smooth map 
on a compact manifold $M$ and $\phi:M \to \mathbb R$ be a continuous potential so that
$\mu_\phi$ is the unique equilibrium state for $f$ with respect to $\phi$. If $\mu_\phi$ is an expanding measure and
$J_{\mu_\phi} f$ is H\"older continuous and has exponential large deviations upper bound then
for any continuous $\psi: M\to \mathbb R$ and $c>0$
$$
P_{\underline{X}_{c} }(f , \phi)
    \leq P_{\topp}(f , \phi) - L_{c -\delta}
    <\Ptop (f,\phi)
$$
for every small $\delta > 0$.
\end{maincorollary}

The key ingredient used in the proof of Corollary~\ref{thm:nue-add1} 
is that hyperbolic times are instants at which the Gibbs property holds provided the Jacobian of the measure
has enough regularity to deduce bounded distortion.
One should also point out that if a local large deviations principle holds as in equations~\eqref{eq:lld1} and~\eqref{eq:lld2}
(e.g. \cite{Young,RY08,Bomfim}) then it is not hard to see that an upper bound for $P_{\underline{X}_{c} }(f , \phi)$
can be taken as
$$
P_{\topp}(f , \phi) - \min\Big\{I_{f,\phi,\psi}\Big(\int\psi d\mu_{\phi} + c\Big) \;,\; I_{f,\phi,\psi}\Big(\int\psi d\mu_{\phi} - c\Big)\Big\}.
$$
There are examples where the right hand side term above can also be shown to vary continuously with
the data even in the non-uniformly expanding context (see e.g. \cite{Bomfim}). Since we only estimated
the topological pressure of the sets $\underline{X}_{c}$ in the non-uniformly expanding context one
question that arises naturally is the following

\vspace{.1cm}
{\bf Question:} Are there examples of transitive non-uniformly expanding maps under the conditions of the previous theorem
where $P_{\overline{X}_{c} }(f , \phi)$ differs from $P_{\underline{X}_{c} }(f , \phi)$ and coincides with the topological pressure $\Ptop(f,\phi)$?
\vspace{.1cm}

We include some examples where we give partial answers to this question in Section~\ref{Examples}
by proving that these sets may have different upper Carath\'eodory capacities.

\section{Proof of the main results}\label{Proofs}

\subsection{Proof of Theorem~\ref{thm:exp-add}}

Our purpose it to estimate $P_{\underline{X}_{I} }(f , \phi)$ and
$P_{\overline{X}_{I} }(f , \phi)$. Consider the sets
$
X_{I,n}=\{x \in M : \frac{1}{n}S_{n}\psi(x) \in I \}.
$
Let us first prove a preliminary lemma.

\begin{lemma}\label{lemma:aux1}
Let $I \subset \mathbb R$ be a closed set. For any $\delta> 0$ there exists $\vep_\delta>0$ and $N=N_\delta \in \N$ so that
$B(x ,n,  \vep) \subset X_{I_{\delta}, n}$ for all $0< \vep < \vep_\delta$, $n \geq N$ and $x \in X_{ I, n}$.
\end{lemma}

\begin{proof}
Let $\delta>0$ be given. Since $\psi$ is uniformly continuous then there is $\vep=\vep_\delta>0$ and
a large $N=N_\delta \in \N$ so that $\gamma_{n}(\psi,\vep)\leq \delta n$ for all $0 <\vep <\vep_\delta$ and
$n \geq N$. So, if $n \geq N$, $x \in X_{ I, n}$, $y \in B(x ,n, \vep)$ and $0<\vep<\vep_\delta$ then
$$
 \frac{S_{n}\psi(x)}{n} -\frac{\gamma_{n}(\psi,\vep)}{n}
 	\leq \frac{S_{n}\psi(y)}{n}
 	\leq  \frac{S_{n}\psi(x)}{n} + \frac{\gamma_{n}(\psi,\vep)}{n}
$$
and, consequently,
$$
\frac{S_{n}\psi(x)}{n} -\delta
	\leq \frac{S_{n}\psi(y)}{n}
	\leq \frac{S_{n}\psi(x)}{n} + \delta
$$
meaning that $y \in X_{I_{\delta}, n}$. This finishes the proof of the lemma.
\end{proof}

\begin{proof}[Proof of Theorem~\ref{thm:exp-add}]
Let $I\subset \mathbb R$ be a closed interval and assume $\overline{X}_{I}$ is non-empty. Let $\delta>0$ be fixed and
consider $L_{I_\delta}$ as defined in equation~\eqref{eq:Lc}. For any positive integer $n$ consider the set $\mathcal I_n\subset M\times
\mathbb N$ of pairs $(x,n)$  with  $x\in M$.
Recalling the notion of topological pressure for invariant sets introduced by Pesin and Pitskel
(see e.g. \cite{pesin}), in order to prove that  $P_{\overline{X}_I}(f , \phi) \leq P_{\topp}(f,\phi) - L_{I_{\delta}}$
it is enough to prove that  for all $\alpha > P_{\topp}(f,\phi) - L_{I_{\delta}} $, every $\vep > 0$ and $N \in \N$
there exists a subset $\hat{\mathcal{G}}_N \subset \bigcup_{n \geq N}\mathcal{I}_n$ so that
$$
\overline{X}_{I}
        \subset \ds\bigcup_{(x,n)\in \hat{\mathcal{G}}_N} B(x,n,\vep)
       \quad\text{and}
       \quad
\ds\sum_{(x , n) \in \hat{\mathcal{G}}_N} e^{-\alpha n +\phi_{n}(x)} \leq a(\vep)< \infty
$$
independently of $N$.

Let $\alpha > P_{\topp}(f,\phi) - L_{I_{\delta}} $ and $0<\vep<\vep_\delta$ be fixed. Notice that if $x \in \overline{X}_{I}$ then there exists a sequence of positive integers $(m_j(x))_{j \in \N}$ converging to infinite with so that
$x \in X_{I_{\delta},m_j(x)}$ for all $j \in \N$. Thus
$
\overline{X}_{I}\subset \bigcap_{\ell\ge 1} \bigcup_{j\ge \ell} X_{I_{\delta},j}.
$
Given $N\ge 1$ and $x \in \overline{X}_{I}$ pick $m(x)\geq N$
in such a way that $x\in X_{I_{\frac{\delta}{2}},m(x)}$ and consider
$\mathcal{G}_N:=\{(x,m(x)):x \in \overline{X}_{I}\}$.
Now, let $\hat{\mathcal{G}}_N \subset \mathcal{G}_N$ be a maximal set with a property of separation, namely,
that if $(x, l)$ and $(y,l)$ belong to
$\hat{\mathcal{G}}_N$ then $B(x,l,\frac{\epsilon}{2}) \cap B(x,l,\frac{\epsilon}{2}) = \emptyset$.
So, for $0 < \vep < \delta$ given by Lemma~\ref{lemma:aux1} using the Gibbs property for $\nu$
we deduce that
\begin{align*}
\sum_{(x,m(x))\in \hat{\mathcal{G}}_N}   e^{-\alpha m(x) +S_{m(x)}\phi(x)}
&=\sum_{(x,m(x))\in \hat{\mathcal{G}}_N}    e^{(P-\alpha)m(x)}e^{-P m(x) +S_{m(x)}\phi(x)}\\
&\leq \sum_{(x,m(x))\in \hat{\mathcal{G}}_N}   e^{(P-\alpha)m(x)}K(\vep)\nu(B(x,m(x),\vep))
\end{align*}
Now, we write  $\hat{\mathcal{G}}_N=\cup_{\ell\ge 1} \hat{\mathcal{G}}_{\ell,N}$ with the level sets
$\hat{\mathcal{G}}_{\ell,N}:=\{(x,\ell)\in \hat{\mathcal{G}}_{N}\}$ and pick $\zeta>0$ small such that
$\alpha > P_{\topp}(f,\Phi) - L_{I_{\delta}} +\zeta$ and
$
\mu \Big(\{x \in \Lambda : \frac{1}{n}S_{n}\psi(x) \in I_{\delta} \}\Big)
    \leq e^{-(L_{I_{\delta}}-\zeta) n}
$
for all $n\ge N$ large. By Lemma~\ref{lemma:aux1} each dynamical ball $B(x,\ell,\vep)$ is contained in
$X_{I_{\delta},\ell}$. Therefore, using that $\nu(B(x,m(x),\vep)) \leq K(\vep) K(\vep/2) \nu(B(x,m(x),\vep/2)$
then
 \begin{align*}
\sum_{(x,m(x))\in \hat{\mathcal{G}}_N}   e^{-\alpha m(x) +S_{m(x)}\phi(x)}
	    & \leq  K(\vep)\sum_{(x,m(x))\in \hat{\mathcal{G}}_N} e^{(P-\alpha)(m(x))}\nu(B(x,m(x),\vep)) \\
	    & = K(\vep)   \sum_{\ell \geq N}e^{(P-\alpha)\ell}
	        \sum_{x \in \hat{\mathcal{G}}_{N,\ell}}  \nu(B(x,\ell,\vep)) \\
	    & \leq  K(\vep)K(\frac{\vep}{2})
	  \sum_{\ell\geq N}e^{(P-\alpha)\ell}
	  \sum_{x \in \hat{\mathcal{G}}_{N,\ell}}  \nu(B(x,\ell,\vep/2)) \\
	  & \leq  K(\vep)K(\frac{\vep}{2})
	  \sum_{\ell\geq N}e^{(P-\alpha) \ell} \nu(X_{I_{\delta},\ell}) \\
	  & \leq
	  K(\vep)K(\frac{\vep}{2}) \sum_{\ell\geq N}e^{(P-\alpha-L_{c-\delta}+\zeta) \ell}
\end{align*}
which is finite and independent by the choice of $\alpha$. This proves that $P_{\overline{X}_{I} }(f , \phi) \leq
P_{\topp}(f,\phi) - L_{I_{\delta}} $. Since  $P_{\underline{X}_{I} }(f , \phi)\leq P_{\overline{X}_{I} }(f , \phi)$
this finishes the proof of the theorem.
\end{proof}

\subsection{Proof of Theorem~\ref{thm:precise.regularity}}

Let us assume that both $\phi,\psi$ satisfy the Bowen condition and $\psi$ is not cohomologous to a constant.
Assume without loss of generality that $\int \psi\, d\mu_{f,\phi}=0$. Our first purpose is to prove
$$
P_{\overline{X}_{c} }(f , \phi)
    \leq
    P_{\topp}(f , \phi) - \min\Big\{I_{f,\phi,\psi}\Big(\int\psi d\mu_{\phi} + c\Big) \;,\; I_{f,\phi,\psi}\Big(\int\psi d\mu_{\phi} - c\Big)\Big\},
$$
where $I_{f,\phi,\psi}$ is the rate function of the large deviations function.
Since $f\mid_\Lambda$ satisfies the specification property and $\psi$ is not cohomologous to a constant it
follows that (see e.g.  \cite{top2})
$$
\Big\{\alpha \in \R : \exists \;x \in M \,\text{s. t.}\,  \lim_{n\to\infty} \frac{1}{n}S_{n}\psi(x) = \alpha \Big\}
     = \Big\{\int \psi d\mu : \mu \, \text{is}\, f\,\text{-invariant}\Big\}
$$
is a non-empty compact interval. By the level-1 large deviations principle for uniformly hyperbolic dynamics
of Young~\cite{Young} equations~\eqref{eq:lld1} and~\eqref{eq:lld2} hold
with the rate function $I_{f,\phi,\psi}(s) = \sup \{-P_{\topp}(f,\phi) + h_{\eta}(f) + \int \phi \, d\eta : \int \psi \,d\eta = s\}$.
Moreover, it follows from the functional analytic approach using transfer operators and the differentiability of the
free energy function that $I_{f,\phi,\psi}$ is the Legendre transform of the free energy.
On the one hand, using Theorem~\ref{thm:exp-add} and the previous upper bound
$$
P_{\overline{X}_{c}}(f , \phi)
	\leq P_{\topp}(f,\phi) - L_{c-\delta}
	\leq P_{\topp}(f , \phi) - \min \{ I_{f,\phi,\psi}(c-\delta), I_{f,\phi,\psi}(c+\delta) \}
$$
for all positive $\delta$. Now, assume for simplicity that $0 < c = c_{1} < c_{2}$
and $c_*=-c$ is defined by equation~\eqref{eqces} (the other cases are analogous).
We claim that it follows from the continuity and convexity of the rate function that
if $X_{c}\neq \emptyset$  then
\begin{align*}
P_{\overline{X}_{c}}(f,\phi)
    & = P_{\underline{X}_{c}} (f,\phi) = P_{X(-c)} (f,\phi) = P_{X([-c_2,-c_1])}(f,\phi)  \\
    & = P_{X(-c_2,-c_1)}(f,\phi)= P_{\topp}(f , \phi) - I_{f,\phi,\psi}(-c_1) \\
    & = P_{\topp}(f , \phi) - I_{f,\phi,\psi}(c_*).
\end{align*}
In fact, using  \cite{top2} the topological pressure of the set
$\{x  \in M : \lim\frac{1}{n}S_{n}\psi(x) = c \}$ coincides with
$
\sup\{h_{\eta} + \int \psi \,d\eta : \eta\; \mbox{is} \;f\mbox{-invariant}\;\text{and}\; \int \psi d\eta = c\}.
$
Then
\begin{align*}
P_{\topp}(f , \phi) - I_{f,\phi,\psi}(-c_{1})
    & =  P_{X(-c_{1})} (f,\phi)
    	\leq P_{X(-c_2, -c_1)}(f,\phi)  \\
        & \leq P_{X[-c_2, -c_1]}(f,\phi)
        	\leq P_{\underline{X}_{c_{1}}} (f,\phi)    \\
    & \leq  P_{\overline{X}_{c_{1}}} (f,\phi)
    \leq P_{\topp}(f , \phi) - I_{f,\phi,\psi}(c_{1}) \\
    & \leq P_{\topp}(f , \phi) - I_{f,\phi,\psi}(-c_{1}).
\end{align*}
This proves the first part of the theorem.
We proceed to prove the continuity results using that $P_{\underline{X}_{f, \phi, \psi, c}} (f,\phi)= P_{\topp} (f, \phi) -  \min \{ I_{f,\phi,\psi}(c), I_{f,\phi,\psi}(-c)\}$ whenever the set $ \underline{X}_{f, \phi, \psi, c}$ is non-empty. On the one hand,
it is well known that $\phi\mapsto P_ {\topp} (f, \phi) $ is continuous in the $C^0$-topology. On the other hand, if $ \Lambda = M $ then $f$ is expanding and $P_{\topp}(f , \phi)$ varies continuously with $f$ in the $C^1$-topology since it coincides
with the logarithm of the spectral radius of the transfer operator $\mathcal L_{f,\phi}: C^\al(M) \to C^\al(M)$ given by
\begin{equation*}\label{eq:transfer}
\cL_{f,\phi} \, g(x)= \sum_{f(y)=x} e^{\phi(y)} \,g(y).
\end{equation*}
In fact, due to the existence of a spectral gap property for $\cL_{f,\phi}$ the spectral radius does vary continuously with respect to perturbations of the potential and the Legendre transform varies continuously with respect to the potential. We will provide a sketch of proof now addressing also the continuity of these objects as function of the dynamics
$f$ and observable $\psi$.

Given $f$, $\phi $ and $ \psi $ fixed, the spectral gap property for $\cL_{f,\phi}$ implies that free energy
$
\cE_{f,\phi,\psi}(t)
    :=\limsup_{n\to\infty} \frac1n \log \int e^{t S_n \psi} \; d\mu_{f,\phi}
$
is well defined for all $t \in \R$ and in fact it verifies $\cE_{f,\phi,\psi}(t)= \Ptop(f, \phi +t\psi) -\Ptop(f, \phi)$.
In particular, if $\psi$ is cohomologous to a constant then $t \mapsto \cE_{f,\phi,\psi}(t)$ is affine and otherwise
$t \mapsto \cE_{f,\phi,\psi}(t)$ is real analytic, strictly convex.  Note also that for every $t \in \R$ the function
$(f, \phi,\psi)\mapsto \cE_{f,\phi,\psi}(t)$ is differentiable, the function $(\phi,\psi)\mapsto \cE_{f,\phi,\psi}(t)$
is analytic and also
$$
(f,\phi,\psi) \mapsto \cE_{f,\phi,\psi}'(t) =\int \psi \, d\mu_{f,\phi+t\psi}
$$
is continuous (see e.g. \cite{Bomfim}).
Let $t\in\mathbb R$ be fixed. In order to establish the regularity of the rate function
$I_{f,\phi,\psi}$ in what follows we assume without loss of generality that $\psi$ is not cohomologous to
a constant and that $m_{f,\phi}=\int \psi \, d\mu_{f,\phi}=0$. Using that  $\R \ni t\to \cE_{f,\phi,\psi}(t)$ is strictly
convex it is well defined its Legendre transform $I_{f,\phi,\psi}$ by
\begin{equation*}
I_{f,\phi,\psi}(s)
    = \sup_{t \in \R} \; \big\{ st-\cE_{f,\phi,\psi}(t) \big\}.
\end{equation*}
This function is non-negative and strictly convex since $\cE_{f,\phi,\psi}$ is also strictly convex, and
$I_{f,\phi,\psi}(s)=0$ if and only if $s=m_{f,\phi}$.
Morever, using the differentiability of the free energy function it is not hard to check the variational property
$$
I_{f,\phi,\psi} (\cE'_{f,\phi,\psi}(t) )
    = t \,\cE'_{f,\phi,\psi}(t) - \cE_{f,\phi,\psi}(t)
$$
whenever the expressions make sense and, consequently, the rate function $I_{f,\phi,\psi}$ varies continuously
with $\phi$ and $\psi$ in the $C^\al$-topology.

Finally we study the regularity of the function $v\mapsto P_{\overline{X}_c}(f_v,\phi)$ when $V \ni v \mapsto (f_{v})_v$  is a continuous family of expanding maps on $M$ and $V$ is a compact metric space.
Let $J \subset \R$ be a compact interval in the domain of $I_{f_v,\phi,\psi}$. From the previous variational relation
we get that for any $s\in J$ there exists a unique $t=t(s,v)$
such that $s=\cE'_{f_v,\phi_{v},\psi_{v}}(t) $ and
\begin{equation}\label{eq.var.rate}
I_{f_v,\phi,\psi} (s)= s \cdot t(s, v) - \cE_{f_v,\phi,\psi}(t(s, v)).
\end{equation}
Now, notice that the skew-product
$$
\begin{array}{ccc}
F: V \times J & \to & V \times \mathbb R\\
(v,t) & \mapsto & (v, \cE'_{f_v,\phi,\psi}(t))
\end{array}
$$
is continuous and injective because it is strictly increasing along the fibers (using the strict convexity
of the free energy function).
Since $V\times J$ is a compact metric space then $F$ is a homeomorphism
onto its image $F(V\times J)$. In particular this shows that for every $(v,s)\in F(V\times J)$ there
exists a unique $t=t(v,s)$ varying continuously with $(v,s)$ such that $F(v,t(v,s))=(v, s)$ and
$s=\cE'_{f_v,\phi_v,\psi_v}(t)$. Finally, relation~\eqref{eq.var.rate} above yields that $(s,v) \mapsto I_{f_v,\phi,\psi}(s)$ is continuous on $J\times V$.  This finishes the proof of the continuity.

\subsection{Proof of Corollary~\ref{cor:thompson2}}

Let $f: M\to M$ be a continuous map admitting a mixing repeller $\Lambda \subset M$,
$\phi:M \to \mathbb R$ be a continuous potential so that $\mu_\phi$ is the unique equilibrium state for
$f$  with respect to $\phi$ and $\mu_\phi \ll \nu$ where $\nu$ is a Gibbs measure. 
Given $c>0$ consider $\overline{E}_c=\overline{X}_c \cap E_\psi \subset \overline{X_c}$.
Taking into account Theorem~\ref{thm:precise.regularity} and Corollary~\ref{cor:thompson} then part (1)
is immediate and we are reduced to prove the lower bound: if $\overline{E}_c \neq \emptyset$ then 
$P_{\overline{E}_c}(f,\phi) \ge P_{\overline{X}_c}(f,\phi)$.  

Assume $\overline{E}_c \neq \emptyset$ for some $c>0$. Under our assumptions it is well known that $f$ satisfies the specification property and that for any 
$f$-invariant probability measure $\mu$ there exists a sequence of $f$-invariant ergodic probability
measures $\mu_n$ so that $\mu_n\to\mu$ in the weak$^*$ topology and $h_{\mu_n}(f) \to h_\mu(f)$
as $n\to\infty$ (c.f. Theorem B in \cite{EKW}). 
By Theorem~\ref{thm:precise.regularity} and the thermodynamical formulation of the large deviations 
rate function obtained by L.S. Young~\cite{Young} 
we know that
\begin{equation}\label{eq:variationalXc}
P_{\overline{X}_c}(f,\phi) 
	= P_{\topp}(f , \phi) - I_{f,\phi,\psi}(c_*)
	= \sup_\eta \Big\{ h_\eta(f) + \int \phi \, d\eta \Big\}
\end{equation}
where $|c_*|=|c|$ and the supremum is taken over all $f$-invariant probability measures $\eta$
so that $| \int \psi \,d\eta -\int \psi \, d\mu_\phi| \ge c_*$. Assume for simplicity that $c_*=c$ (the case other is analogous). Observe that 
\begin{align*}
\sup & \Big\{ h_\eta(f) + \int \phi \, d\eta  : | \int \psi \,d\eta -\int \psi \, d\mu_\phi| \ge c\Big\} \\
	& = \sup \Big\{ h_\eta(f) + \int \phi \, d\eta  : | \int \psi \,d\eta -\int \psi \, d\mu_\phi| > c\Big\}
\end{align*}
by the continuity of the rate function $c \mapsto I_{f,\phi,\psi}(c)$ (since it coincides with the 
Legendre transform of the free energy function).
Together with the variational relation \eqref{eq:variationalXc}, this yields that for any $\gamma>0$ 
one can take two $f$-invariant probability measures $\eta_1, \eta_2$ so that 
\begin{enumerate}
\item[(i)] $| \int \psi \,d\eta_i -\int \psi \, d\mu_\phi| > c$ 
\item[(ii)] $h_{\eta_i}(f) + \int \phi \, d{\eta_i} \ge P_{\overline{X}_c}(f,\phi) -2\gamma$ 
\item[(iii)] $\int \psi\, d\eta_1 \neq \int \psi\, d\eta_2$
\end{enumerate}
for $i=1,2$. Taking the approximation in entropy by $f$-invariant and ergodic probability measures, 
there are distinct ergodic probability measures $\nu_1$ and $\nu_2$ satisfying 
$| \int \psi \,d\nu_i -\int \psi \, d\mu_\phi| > c$, $\int \psi\, d\nu_1 \neq \int \psi\, d\nu_2$ and
$h_{\nu_i}(f) + \int \phi \, d{\nu_i} \ge P_{\overline{X}_c}(f,\phi) -\gamma$ for $i=1,2$.
Observe that $\overline{X}_{c}$ is an $f$-invariant set and the ergodicity together with the first property 
above implies that  $\nu_i(\overline{X}_{c})=1$.

Now the proof follows the same lines of the proof of Theorem~2.6 in \cite{Daniel}. 
Consider a strictly decreasing sequence $(\delta_k)_{k\ge 1}$ of positive numbers converging
to zero, a strictly increasing sequence of positive integers $(\ell_k)_{k\ge 1}$, so that the sets 
$$
Y_{2k+i}=\Big\{ x\in \overline{X}_c \colon | \frac1n S_n\psi(x)  -  \int \psi \, d\nu_{i}| < \delta_k 
	\text{ for every } n\ge \ell_k \Big\}
$$
satisfy $\nu_i(Y_{2k+i}) > 1-\gamma$ for every $k$ ($i=1,2$).
Consider the fractal set $F$ given \emph{ipsis literis} by the construction of Subsection~3.1 with
$\nu_i$ replacing $\mu_i$,  $P_{\overline{X}_c}(f,\phi)$ replacing $C$  and $\psi$ replacing $\varphi$.
From the construction (c.f. Lemma~3.8) there is a sequence $(t_k)_{k\ge 1}$ so that
$$
\lim_{k\to\infty} | \frac{1}{t_{2k+i}}  S_{t_{2k+i}} \psi (x) - \int \psi \, d\nu_i | = 0
	\quad \text{for every $x\in F$}
$$
and $P_F(f,\psi) \ge C-8 \gamma$.
In particular $F$ is contained in the irregular set $E_\psi$.
Since $\gamma$ was chosen arbitrary and $F \subset E_\psi$, 
to complete the proof of the corollary it is enough to prove that $F \subset \overline{X}_c$.
This actually follows from item (1) above since for any $x\in F$ 
\begin{align*}
\limsup_{k\to\infty}   \Big| \frac1{t_{2k+2}} S_{t_{2k+2}} \psi(x) -\int \psi \,d\mu_\phi \Big| 
	& \ge \limsup_{k\to\infty} \Bigg[ \Big| \int \psi \,d\nu_2 - \int \psi \,d\mu_\phi \Big|
		  \\
	&  - \Big| \frac1{t_{2k+2}} S_{t_{2k+2}} \psi(x) -\int \psi \,d\mu_\phi \Big| \Bigg]\\
	& \ge c
\end{align*}
This finishes the proof of the corollary.

\subsection{Proof of Theorem~\ref{thm:level-2}}

For the proof of the theorem we will need the following auxiliary lemma that will play the same role
of Lemma~\ref{lemma:aux1} in the proof of Theorem~\ref{thm:exp-add}. It is here that we need the metric
on $\mathcal M_1$ to be translation invariant and affine.

\begin{lemma}\label{lemma:aux2}
Let $c>0$ be given. For any $\delta> 0$ there exists $\vep_\delta>0$ and $N=N_\delta \in \N$ so that
$B(x ,n,  \vep) \subset Y_{\mu,c - \delta, n}$ for all $0< \vep < \vep_\delta$, $n \geq N$ and $x \in Y_{\mu, c, n}$.
\end{lemma}

\begin{proof}
Since $M \in x \mapsto \delta_{x} \in \mathcal{M}_{1}$ is uniformly continuous then given $\delta > 0$ there exists
$\vep_{\delta} > 0$ such that if $d(x, y) < \vep_{\delta}$ we have $d(\delta_{x},\delta_{y}) < \delta$. Hence, if
$x \in Y_{ c, n}$ and $y \in B(x,n,\vep)$ we have:
$
d(\delta_{y,n},\mu)
	 \geq d(\delta_{x,n},\mu) - d(\delta_{x,n},\delta_{y,n})
	 \geq c- \frac{1}{n}\sum_{i=0}^{n-1}d(\delta_{x},\delta_{y})
	 \geq c -\delta,
$
 and thus $y \in Y_{\mu, c,n}$, which proves the lemma.
\end{proof}

We proceed with the proof of the theorem assuming that $\mu=\mu_{f,\phi}$ is the unique equilibrium state for the continuous map $f$ with respect to the continuous potential $\phi$ and also $\overline{Y}_{\mu, c} \neq\emptyset$.
In order to prove that
$
P_{\overline{Y}_{\mu,c} }(f , \phi) \leq P_{\topp}(f , \phi) - \inf_{d(\eta ,\mu) = c}Q(\eta)
$
is strictly smaller than the topological pressure $P_{\topp}(f , \phi) $ we proceed to cover $\overline{Y}_{\mu, c}$
by a properly chosen family of dynamical balls. Fix $\delta>0$ small and
$\alpha> P_{\topp}(f , \Phi) - \inf_{ d(\eta ,\mu) \geq c-\delta}Q(\eta)$.
Given $\vep > 0$ small and $N \in \N$, for any $x \in \overline{Y}_{\mu,c}$ pick $m(x)\geq N$ in such a way that
$x\in Y_{\mu,c-\frac{\delta}{2},m(x)}$ and consider $\mathcal{G}_N:=\{(x,m(x)):x \in \overline{Y}_{\mu,c}\}$.
Hence
$$
\overline{Y}_{\mu,c}
        \subset \ds\bigcup_{(x,n)\in \mathcal{G}_N} B(x,n,\vep)
$$
and also $B(x, n ,\vep) \subset \overline{Y}_{\mu,c-\delta,n}$, for all $x \in Y_{\mu,c-\frac{\delta}{2},n}$ and $n \geq N$ and $\vep$ small (by Lemma~\ref{lemma:aux2}).
Therefore we can proceed as in the proof of Theorem~\ref{thm:exp-add} and extract a subset $\hat{\mathcal{G}}_N
\subset \mathcal{G}_N$ in such a way that if $(x, l)$ and $(y,l)$ belong to
$\hat{\mathcal{G}}_N$ then $B(x,l,\frac{\vep}{2}) \cap B(x,l,\frac{\vep}{2}) = \emptyset$.
If $\zeta= (-P+\alpha + \inf_{ d(\eta ,\mu) \geq c-\delta}Q(\eta))/2>0$,
by the large deviations upper bound, for every closed set $U$ one has $\mu \left(x\in M : \delta_{x,n} \in U \right) \le \exp (-n [\inf_{\eta \in U} Q(\eta)
-\zeta])$ provided that $n\ge N$ is large enough.
This, together with the Gibbs property for
$\nu$, yields that
\begin{align*}
\sum_{(x,n)\in \hat{\mathcal{G}}_N} e^{-\alpha n +S_n\phi(x)}
    &\leq K(\vep)\sum_{n\ge N} \sum_{x\in \hat{\mathcal{G}}_{N,n}} e^{(P-\alpha) n }\nu(B(x,n,\vep)) \\
    & \leq K(\vep) K(\frac{\vep}{2})  \sum_{n\ge N} e^{(P-\alpha) n}\nu\Big( \bigcup_{x\in \hat{\mathcal{G}}_{N,n}}
    		B(x,n,\frac{\vep}{2})\Big) \\
    & \leq K(\vep) K(\frac{\vep}{2})  \sum_{n\ge N} e^{(P-\alpha) n}\nu\big( \overline{Y}_{\mu,c-\delta,n}\big) \\
    & \leq  K(\vep) K(\frac{\vep}{2})  \sum_{n\ge N} \exp n
        \big(
         P-\alpha - \inf_{d(\eta ,\mu) \geq c-\delta}Q(\eta) + \zeta
        \big)\\
    & \leq  K(\vep) K(\frac{\vep}{2})  \sum_{n\ge N} e^{-\zeta n}
\end{align*}
which is finite and independent of $N$. This proves that
$
P_{\overline{Y}_{\mu,c}}(f,\phi)
    \leq P_{\topp}(f , \phi) - \inf_{ d(\eta ,\mu) \geq c-\delta}Q(\eta).
$
Since $Q$ is lower semicontinuous it follows that
$$
P_{\overline{Y}_{\mu,c}}(f,\phi)
    \leq P_{\topp}(f , \phi) - \inf_{ d(\eta ,\mu) \geq c}Q(\eta).
$$

For the proof of the second part of the theorem we make use of the level-2 large deviations principles obtained
by Zhou and Chen~\cite{ZC13} under the assumptions of almost product structure
and the uniform separation properties.
Let $0 < c_1 < c_{2}$ be so that $\overline{Y}_{\mu,c_1}\neq\emptyset$. Using \cite{ZC13}, given a compact connected subset in $\mathcal{M}_{1}$ then the topological pressure of the set
$Y(C) := \{x \in M : \lim_{n \to +\infty}\delta_{x,n} \in C \}$
coincides with
$
\inf \{h_{\eta}(f) + \int \psi \,d\eta : \eta\; \mbox{is} \;f\mbox{-invariant}\;\text{and}\; \eta \in C\}.
$
On other hand, by \cite{CTY13}
$Q(\eta) = P_{\topp}(f,\phi) - h_{\eta}(f) - \int\phi d\eta$. Thus using that the metric entropy is linear convex and
the choice of the metric on $\mathcal{M}_{1}$ we have
\begin{align*}
P_{\topp}(f , \phi) - \inf_{d(\eta,\mu) = c_{1}}Q(\eta)
	& = P_{Y(\partial B(\mu,c_{1}))} (f,\phi)
      \leq P_{Y(B(\mu,c_{1},c_{2}))}(f,\phi)\\
    & \leq  P_{Y\big(\overline{B(\mu,c_{1}, c_{2})}\big)}(f,\phi)
         \leq P_{Y(B(\mu,c_{1}))} (f,\phi)    \\
    & \leq  P_{Y(\overline{B(\mu,c_{1})})} (f,\phi)
     \leq P_{\overline{Y}_{\mu,c_1}}(f,\phi) \\
    & \leq P_{\underline{Y}_{\mu,c_1}} (f,\phi)
     \leq P_{\topp}(f , \phi) - \inf_{d(\eta ,\mu) \geq c_{1}}Q(\eta) \\
    & \leq P_{\topp}(f , \phi) - \inf_{d(\eta ,\mu) = c_{1}}Q(\eta),
\end{align*}
proving all quantities coincide. This finishes the proof of the theorem.

\subsection{Proof of Theorem~\ref{thm:nue-add}}

This section is devoted to the proof of Theorem~\ref{thm:nue-add} that claims that if the set
$\underline{X}_{c}$ is non-empty then it has smaller topological pressure. The strategy for the
proof is similar to the one of Theorem~\ref{thm:exp-add} with the difficulty that in the non-uniformly
expanding setting the Gibbs property holds at a sequence of moments that does depend on the point.
For that reason we shall give a sketch of the proof with the main ingredients.
Since $\mu$ is a weak Gibbs measure then there exists $\vep_0>0$ so that the following property
holds: for every $0<\vep<\vep_0$ there exists $K(\vep)>0$ and  for $\mu$-almost every $x$ there exists
a sequence $n_k(x)\to \infty$ such that
$$
K(\vep)^{-1} \leq \frac{\mu(B(x,n_{k}(x),\vep))}{e^{-n_{k}(x)P + S_{n_{k}(x)}\phi(x)}} \leq K(\vep).
$$
Assume the weak Gibbs property holds for \emph{all} points in the invariant set $\Lambda=H$
and in what follows consider $\underline{X_I}:=\underline{X_I}\cap \Lambda$.

Let $\delta>0$ be arbitrary. We proceed to prove that  $P_{\underline{X}_{I}}(f , \phi)
\leq P_{\topp}(f,\phi) - L_{I_{\delta}}$ is strictly smaller than the topological pressure.
Consider $\alpha > P_{\topp}(f,\Phi) - L_{I_{\delta}} $ be given and take $\vep > 0$
arbitrarily small  and $N \in \N$ arbitrarily large in what follows. One can write
$$
\underline{X}_{I}\subset \bigcup_{\ell\ge 1} \bigcap_{j\ge \ell} X_{I_{\delta},j}.
$$
where as before $X_{I,n}=\{x \in M : \frac{1}{n}S_{n}\psi(x) \in I\}$.
It is not hard to check that for any $x \in \underline{X}_{I}$ there exists a sequence of positive
integers $(m_j(x))_{j \in \N}$ converging to infinite so that $x \in X_{I_{\delta},m_j(x)}$ and
$m_j(x)$ is a moment at which the Giibs property holds.
Therefore, one can pick   $m(x)\geq N$ in such a way that
$x\in X_{I_{\delta},m(x)-1}$ and consider $\mathcal{G}:=\{(x,m(x)):x \in \overline{X}_{I}\}$.
Now the proof proceeds with the estimates used in the proof of Theorem~\ref{thm:exp-add}.

\begin{remark}
In the previous proof we did not require the times at which the Gibbs property hold to have positive density
at infinity as in usual notions of non-lacunary Gibbs measures. In particular, this gives a wider range of
applications.
\end{remark}

\begin{remark}
Actually let us mention that we could not estimate the topological pressure of the larger
set $\overline{X}_{c}$. In fact, for that purpose we would need to guarantee that for each
point there would exist a sequence of instants at which simultaneously the Gibbs property and
the time averages being far from the time average occurs.
Nevertheless this can be verified in examples.
\end{remark}

\section{Examples and applications}\label{Examples}

\subsection{Hyperbolic diffeomorphisms and flows}

Using the fact that hyperbolic sets admit Markov partitions (c.f. \cite{Bow75}) then same results as
in Theorem~\ref{thm:exp-add} also hold in this hyperbolic setting via semi conjugation to bilateral shifts
and Remark~\ref{rmk:hyperbolic}. 
Axiom A flows $(Y_t)_t$ are also semi-conjugate to suspension flows over subshifts of finite type, by \cite{BR75}. Recall that given a subshift of finite type $\sigma \colon \Sigma\rightarrow{\Sigma}$ and a ceiling function $h\colon \Sigma\rightarrow{\mathbb{R}^{+}}$ bounded away from zero
and infinity the associated \emph{suspension flow} $(S_t)_t$ is defined in
$
\Sigma_h=\{(x,t) \in \Sigma\times{\mathbb{R}_+}: 0 \leq t \leq h(x) \}
$
with the identification between the pairs $(x,h(x))$ and $(\sigma(x),0)$. The semiflow defined
on $\Sigma_h$ by $S_t(x,r)=(\sigma^{n}(x),r+t-\sum_{i=0}^{n-1} h(\sigma^{i}(x)))$,
where $n=n(x,r+t) \in{\mathbb N}_0$ is uniquely defined by 
$\sum_{i=0}^{n-1}h(\sigma^{i}(x))\leq{r+t}<\sum_{i=0}^{n}h(\sigma^{i}(x))$.
It is clear that 
$
\frac1{n(x,T+s)} \sum_{i=0}^{n(x,T+s)-1}h(\sigma^{i}(x))
	\leq \frac{T +s}{n(x,T+s)}
	< \frac1{n(x,T+s)} \sum_{i=0}^{n(x,T+s)}h(\sigma^{i}(x))
$
and $n(x,T+s) \to \infty$ as $T\to\infty$.
Given $\psi \in C(\Sigma_{h}, \mathbb R)$ define $\bar \psi \in C(\Sigma,\mathbb R)$ by $\bar\psi(x)=\int_0^{h(x)} \psi(x,t) \;dt$.
Let $\mu$ and $\mu_\Sigma$ be the unique equilibrium states for $(X_t)_t$ with respect to $\phi$ and 
for $\sigma$ with respect to $\overline{\phi}$, respectively, which are known to satisfy 
$\mu=\mu_\Sigma\times \text{Leb}_1 / \int h d\mu_\Sigma$ (c.f. \cite{BR75}).
If $\beta_{1}, \beta_{2}, \beta_{3} >0$ are small 
it is not hard to check that
\begin{align*}
\overline{X}_c
	=\Big\{ (x,s) \in \Sigma_h  : \limsup_{T\to \infty} \big| \frac1T \int_0^T \psi(S_t(x,s)) \;dt - \int \psi \;d\mu \big| \geq c \Big\}
\end{align*}
is contained in the union of the $(S_t)_t$-invariant sets
\begin{align*}
\overline{X}_{c,h}
	& =\big\{ (x,s) \in \Sigma_h : \limsup_{T\to\infty} \big| \frac{n(x,T+s)}{T+s} -\frac{1}{\int h \;d\mu_\Sigma} \big| \geq \beta_{2} \big \} \\
	& \subseteq \big\{ (x,s) \in \Sigma_h :
		\limsup_{T\to\infty} \big|  \frac1{n(x,T+s)} \sum_{i=0}^{n(x,T+s)}h(\sigma^{i}(x)) -\int h \;d\mu_\Sigma \big| \geq
			\beta_{3}\big \}
\end{align*}
and
\begin{align*}
\overline{X}_{c,\overline\psi}
	& =\big\{ (x,s) \in \Sigma_h :  \limsup_{T\to\infty}
	\big|
	\frac1{n(x,T)} \sum_{i=0}^{n(x,T+s)} \overline \psi (\sigma^{i}(x)) -\int \overline \psi \; d\mu_\Sigma\big|
		\geq \beta_{1}
	 \big\}
\end{align*}
Using
$
P_{\overline{X}_c} ((S_t)_t,\phi)
	\leq \max\{ P_{\overline{X}_{c,h}} ((S_t)_t,\psi), P_{\overline{X}_{c,\overline\phi}}   ((S_t)_t,\phi) \}
$
and Theorem~\ref{thm:exp-add} we deduce by semi-conjugacy that $P_{\overline{X}_{c} }((X_t)_t , \phi)<\Ptop ((X_t)_t,\phi)$
for any
Axiom A flow, any H\"older potential $\phi:M \to \mathbb R$,
any continuous $\psi: M\to \mathbb R$ and  $c>0$.

\subsection{Maneville-Pommeau maps}

If $\al \in (0,1)$, let $f_\alpha:\mathbb S^{1}\to \mathbb  S^{1}$ be the local
homeomorphism given by $f_\al(x)=x(1+2^{\alpha} x^{\alpha})$ for $0 \leq x \leq \frac{1}{2}$
and by $f_\al(x)=2x-1$ whenever $\frac{1}{2} < x \leq 1$.
This map satisfies the specification property since it is topological conjugate to the double expanding map.
Pollicott and Weiss~\cite{PW99} established a multifractal formalism for the Lyapunov spectrum associated to
this class of transformations and proved precise formulas for the dimension of the level sets of points with same
Lyapunov exponent. Clearly
$
\frac{1}n \log |(f_\al^n)'(x)|
    =\frac1n \sum_{j=0}^{n-1} \psi( f_\al^j(x))
$
with $\psi(x)=\log |(f_\al)'(x)|$.
For every $t\in (-\infty, 1)$ there exists a unique equilibrium state $\mu_t$ with respect
to the H\"older continuous potential $\phi_t=-t\log |(f_\al)'(x)|$ and it is well known that
there are two equilibrium states for $f_\al$ with respect to $-\log |(f_\al)'(x)|$ namely
an acip $\mu_1$ and the Dirac measure $\delta_0$. 
Moreover, for every $t \le 1$ the equilibrium state $\mu_t $ for $f$
with respect to the potential $\phi_t$ satisfies a weak Gibbs property: there are constants $K_n$ so that
$\limsup_{n\to\infty}\frac1n\log K_n=0$ and
$$
{K_n}^{-1} e^{-n P_t}|(f^n)'(x)|^t
    \leq \mu_t(\cP^{(n)}(x))
    \leq K_n e^{-n P_t}|(f^n)'(x)|^t
$$
for all $x\in [0,1]$ and $n\ge 1$, where $\cP$ is the Markov
partition for $f$,  $\cP^{(n)}(x)$ is the element of the partition $\cP^{(n)}=\bigvee_{j=0}^{n-1} f^{-j}\cP$ that
contains $x$ and $P_t=\Ptop(f,\phi_t)$. By the Ruelle inequality all measures $\mu_t$ are expanding.
Our results do not apply for $\mu_1$ since polynomial upper and lower bounds for H\"older continuous
observables have been established in~\cite{MN08,Mel09, PS09}.
If $|t|$ is small, $c>0$ and $\psi$ is  \emph{continuous} then
\begin{align*}
 \limsup_{n\to \infty} \frac1n & \log \mu_t \left(x \in M : \left|\frac1n S_n \psi(x)-\!\int \psi\,d\mu_t \right|
    \ge c\right) <0
\end{align*}
(c.f. \cite{Va12}). If in addition we assume $\psi$ is H\"older continuous then there exists an interval $J\subset \mathbb R$ such that the following
local large-deviations principle holds
$$
\limsup_{n\to\infty} \frac1n \log \mu_{t}
    \left(x\in M : \frac1n S_n\psi(x) \in [a,b] \right)
    \le-\inf_{s\in[a,b]} I_{f,\phi_t,\psi}(s)
$$
and
$$
\liminf_{n\to\infty} \frac1n \log \mu_{t}
    \left(x\in M : \frac1n S_n\psi(x) \in (a,b) \right)
    \ge-\inf_{s\in (a,b)} I_{f,\phi_t,\psi}(s)
$$
every $[a,b]\subset J$, and continuity and smoothness of the rate function are also obtained  (c.f.\cite{Bomfim}).
As a consequence we deduce from Theorem~\ref{thm:nue-add} that for all $c>0$ satisfying
$[\int \psi d\mu_{t} - c, \int \psi d\mu_t + c]\subset J$ either $\underline{X}_{c} = \emptyset$ or else
\begin{align*}
P_{\overline{X}_c}(f,\phi_t)
    & = P_{\underline{X}_{c}} (f,\phi_t) = P_{X(c)} (f,\phi_t) = P_{X([c_{1},c_{2}])}(f,\phi_t)  \\
    & = P_{X(c_{1},c_{2})}(f,\phi_t)= P_{\topp}(f , \phi_t) - I_{f,\phi_t,\psi}(c).
\end{align*}
for $c=\max\{|c_{1}|,|c_{2}|\}$.
Furthermore, the right hand side expression varies continuously with $c$ and also varies continuously
with $f$, $\phi_t$ and $\psi$. Although the set of irregular points has full topological entropy
$\log 2$ (see e.g. \cite{Daniel})
the set of Lyapunov irregular points whose Birkhoff averages remain far from $\lambda(\mu_0):=\int \psi d\mu_0$
for all large iterates has topological entropy strictly smaller than $\log 2$.

\subsection{Multimodal maps}

Our results also apply to a broad class of transitive multimodal interval maps $f$ with finitely many non-degenerate critical points with negative Schwarzian derivative considered in \cite{BT09}.  If there exists $C>0$ and $\beta>2\ell -1$ so that $|Df^n(c)| \ge n^\beta$ for every critical point $c$ and all $n\ge 1$ (where $\ell$ denotes the maximal order of the critical points) then it follows from \cite[Theorem~1]{BT09} that  there exists $t_1<1$ so that for all $t\in (t_1,1)$:
(i) there exists a unique equilibrium state $\mu_t$ for $f$ with respect to the potential $\varphi_t=-t\log|Df|$; 
(ii) $\mu_t$ has a compatible inducing scheme with exponential tails, hence it has exponential decay of correlations; and (iii) $\mu_t$ has positive Lyapunov exponent almost everywhere.
Moreover, there is a conformal probability measure $\nu_t$ so that $J_{\nu_t} f(x)=e^{P(t)}|f'(x)|^{-t}$
almost everywhere
and $\mu_t\ll \nu_t$, where $P(t)=\Ptop(f,-t\log|f'|)$. In addition, since $\mu_t$ has only positive Lyapunov exponents then almost every point has infinitely many hyperbolic times. If $n$ is a hyperbolic time for $x$ then the Jacobian
$J_{\nu_t} f^n$ has bounded distortion and, consequently, $\nu_t$ satisfies the weak Gibbs property. Moreover
by property (ii) above with the results by \cite{ALFV} 
one has that $\mu_t$ has exponential large deviations and so it satisfies the assumptions of
Theorem~\ref{thm:nue-add}.

\subsection{Higher dimensional non-uniformly expanding maps}

Assume that  $M$ is a compact metric space where the Besicovitch covering lemma
and let  $f:M \to M$ be a local homeomorphism so that: there exists a bounded function
$x\mapsto L(x)$ such that, for every $x\in M$ there is a
neighborhood $U_x$ of $x$ so that $f_x : U_x \to f(U_x)$ is
invertible and
$
d(f_x^{-1}(y),f_x^{-1}(z))
    \leq L(x) \;d(y,z)
$
for every  $y,z\in f(U_x)$.
Assume also that every point has finitely many preimages and that
the level sets for the degree $\{x : \#\{f^{-1}(x)\}=k\}$ are
closed. Given $x\in M$ set $\deg_x(f)=\# f^{-1}(x)$ and 
define
$
h(f)=\liminf_{n\to\infty} \frac1n \log [\min_{x\in M} \deg_x(f^n)].
$
Assume that every point in $M$ has at least $e^{h(f)}$ preimages by
$f$, that $f$ is uniformly expanding outside $\cA$
and not too contracting inside $\cA$ (see \cite{Varandas2} for precise statements).
A concrete example can be build on the torus by taking a linear expanding map $f_0:\torus^d\to\torus^d$,
fixing a small open covering $\cP$ of $\torus^d$ and by deforming $f_0$ on a small neighborhood of a fixed
point $p$ inside $P_1 \in \mathcal P$ by a pitchfork bifurcation in such a way that $p$
becomes a saddle for the perturbed local homeomorphism $f$ (c.f. \cite{Varandas2}).
Then \cite{Va12,Bomfim} imply that local large deviation estimates hold for
all equilibrium states associated to H\"older continuous potentials with low variation
and Theorem~\ref{thm:nue-add} holds in this context.

\subsection{Bowen-eye like systems and a counter-example}

\subsubsection{Distinction of $\overline{X}_c$ and $\underline{X}_c$}

We shall present a simple example of a discrete dynamical system $f$, potential $\phi$, observable
$\psi$ and constant $c>0$ so that $\overline{X}_c \neq \underline{X}_c$.
The map $f$ corresponds
to the time-one map of a flow known as the Bowen eye. The map $f$ has three fixed points
$p_1, p_2$ and $p_3$ (labeled from the left in Figure~\ref{Figure3} below) and is such that
$\{p_1, p_2, p_3\} = \text{Per}(f)=R(f)$ while the non-wandering set is formed by
the fixed point $p_2$ and the closure $D$ of the two separatrices corresponding to the singularities $p_1$ and $p_3$ of the original vector field.
%
\begin{figure}[h]
\includegraphics[scale=.4]{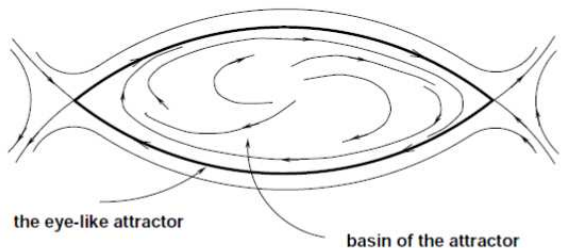}
\vspace{-.3cm}
\caption{Bowen eye attractor}
\label{Figure3}
\end{figure}
%
Moreover, it is well known that for \emph{every} $x$ in inner region of the plane determined by $D$ (except $p_2$)
the empirical  measures $\frac{1}n \sum_{j=0}^{n-1} \delta_{f^j(x)}$ have the Dirac measures $\delta_{p_1}$
and $\delta_{p_3}$ as accumulation points. If $\phi: \mathbb R^2 \to \mathbb R$ denotes the projection on the $x$-coordinate
then, by the variational principle,
$
\Ptop(f,\phi) = \sup \{h_\mu(f) +\int \phi \, d\mu\} 
=\phi(p_3)
$
and $\delta_{p_3}$ is the unique equilibrium state for $\phi$.
On the other hand, for $0<c<d(p_2,p_3)$ it is clear that
$
\underline X_c =W^s(p_1) \cup \{ p_2\}
$
and
$
\overline X_c =D\setminus (W^s(p_3) \cup \{p_2\}).	
$
However, in this case one has $P_{\underline X_c} (f,\phi) = P_{\overline X_c} (f,\phi)$.

\subsubsection{A counter-example}

Despite the fact that the topological pressure of both sets $\underline X_c$ and $\overline X_c$ do
coincide, the previous Bowen-eye construction gives some light on how to construct an example
where  $\overline{CP}_{\underline X_c} (f,\phi) < \overline{CP}_{\overline X_c} (f,\phi)$, where
$\overline{CP}_\Lambda$ denotes the upper Carath\'eodory capacity of the set $\Lambda$
(see e.g. \cite[Section~11]{pesin}). The following can be realized
as a non-compact invariant set of a horseshoe.
Let $\sigma: \Sigma_A \to \Sigma_A$, with $\Sigma_A \subset \{0,1,2,3\}^{\mathbb Z}$, be the subshift of finite type associated to the transition matrix
$$
A=
\left(
\begin{array}{cccc}
 1 & 1 & 1 & 0 \\
 1  & 1 & 1 & 0 \\
 1 & 1 & 1 & 0 \\
 0 & 0 & 0 & 1
\end{array}
\right).
$$
Now, consider the $\sigma$-invariant subset $\Sigma \subset \Sigma_A$ that contains the four fixed points for
the shift $\sigma$ and (corresponding to the constant sequences) and be such that any $x=(x_n)_n \in \Sigma
\setminus \{\underline{3}\}$ it holds that
$\limsup_{n\to\infty} \frac1{2n} \# \big\{|j| \le n : x_j \in \{1,2\} \big\} =1$
 and $\liminf_{n\to\infty} \frac1{2n} \# \big\{-n \le j \le n : x_j \in \{1,2\} \big\} =0$.

Let $\phi$ be a continuous potential so that the unique equilibrium state is $\mu_\phi=\delta_{\underline{0}}$
(such a potential can be build non-negative following the ideas of Hofbauer \cite[Page 226, 239]{Hof77})
and consider the continuous observable $\psi=\chi_{[0]}$. Notice that $\int \psi\,d\mu_\phi=1$ and
for $c>0$ small enough we get that $\underline{X}_c=\{\underline{3}\}$ while
$
\overline{X}_c=\Sigma \setminus \{\underline{0} \} .
$
Since $\phi\mid_{\underline{X}_c}\equiv 0$ and $\underline{X}_c=\{\underline{3}\}$ then 
$\overline{CP}_{\underline{X}_c}(f,\phi)=h_{\underline{X}_c}(f)=0$. On the other
hand, since $\phi$ is non-negative then $\overline{CP}_{\overline{X}_c}(f,\phi)\ge 
\overline{CP}_{\overline{X}_c}(f,0)$ which we now claim to be strictly positive.
In fact if $0<\alpha<\log 2$ we will prove that $m_\alpha(f,\overline{X}_c)=+\infty$ and deduce that
$\overline{CP}_{\overline{X}_c}(f,0)>0$.
Recall that
$
m_\alpha(f,\overline{X}_c)=\lim_{\diam(\mathcal U) \to 0} m_\alpha(f,\overline{X}_c, \mathcal U)
$
where $m_\alpha(f,\overline{X}_c, \mathcal U)= \lim_{N\to\infty} m_\alpha(f,\overline{X}_c, \mathcal U, N)$,
and
$$
m_\alpha(f, \overline{X}_c, \mathcal U, N)
	= \inf \Big\{ \sum_{U \in \mathcal G_N} e^{-\al N}  :
		\mathcal G_N \text{ is subcover of } \vee_{0\le j \le N} \sigma^{-j} {\mathcal U} \Big\}.
$$
Let $\vep>0$ be small and fixed (to be made precise later and depending only on $\al$).
For any $\ell\ge 1$, let us consider an open cover $\mathcal U_\ell$ of $\overline{X}_c$  formed by cylinders as follows:
a $(2n+1)$-cylinder $U=[x_{-n}, \dots, x_n]$ belongs to $\mathcal U_\ell$ if and only if $n\ge \ell$ is the smallest positive integer
such that
\begin{equation}\label{eq10}
\# \big\{|j| \le n : x_j \in \{1,2\} \big\}  \ge (2n+1)(1-\vep).
\end{equation}
In fact, given $x=(x_j)_j \in \underline{X}_c$ and defining $n(x)\ge \ell$ to be the first instant such that equation~\eqref{eq10} holds it follows that $[x_{-n}, \dots, x_n]$ belongs to $\mathcal U_\ell$ and so $\mathcal U_\ell$ covers $\underline{X}_c$.
Moreover, by construction, the elements of $\mathcal U_\ell$ are all disjoint, every such element contains at least one point of $\underline{X}_c$ and the diameter of $\mathcal U_\ell$ goes to zero as $\ell\to\infty$.
Thus $m_\alpha(f,\overline{X}_c)=\lim_{\ell \to\infty} m_\alpha(f,\overline{X}_c, \mathcal U_\ell)$.
Observe also that $\# \mathcal \, U_\ell \ge 2^{(2\ell+1) (1-\vep)}$ which correspond to the number of disjoint cylinders of length $(2\ell+1)$ satisfying \eqref{eq10}. Therefore, given 
any $N\gg1$ and  
any subcover $\mathcal G_{N,\ell}$ of the space of cylinders $\vee_{0 \le j \le N} \sigma^{-j} {\mathcal U_\ell}$
that covers $\overline{X}_c$ coincides with the space of all $(N+\ell)$-cylinders. If one writes
$N+\ell= (2\ell+1) s +r$ with $s\ge 1$ and $0\le r \le 2\ell$ then there are at least $2^{(N+\ell-r)(1-\vep)}$ such cylinders 
(just by considering  $N$-concatenations of $(2\ell+1)$-cylinders that satisfy equation~\eqref{eq10}).
Thus, if $\vep>0$ is chosen small then it follows that 
\begin{align*}
m_\alpha(f,\overline{X}_c, \mathcal U_\ell) 
	 \ge \limsup_{N\to\infty} \sum_{U \in \mathcal G_{N,\ell}} e^{-\alpha N} 
	 \ge  \limsup_{N\to\infty} e^{-\alpha N} \; 2^{(N-\ell)(1-\vep)}
	= +\infty.
\end{align*}
Hence $\overline{CP}_{\overline{X}_c}(f,0)\ge \log 2>0$  which proves our claim.
Let us mention it is still a question to construct 
an example where $P_{\underline X_c} (f,\phi) < P_{\overline X_c} (f,\phi)$. 

\subsection{Discontinuity and non-strict monotonicity of the pressure function}

\subsubsection{Porcupine-like horseshoes}

To present an example where there is discontinuity and non-strict monotonicity of the pressure function
$
c\mapsto P_{\overline{X}_c} (f,\phi)
$
we use the class of local diffeomorphisms $f$ studied by D\'iaz, Gelfert and Rams that exhibit porcupine-like horseshoes.
In fact,  it follows from the analysis of the Lyapunov spectrum in the central direction (see \cite[Remark~5.4]{DG12} and
\cite{DGR11})  that there are constants $\lambda < 0< \tilde\beta <\beta$ so that
$$
\liminf_{n\to\infty} \frac1n \log \|Df^n\mid_{E^c}(x)\| \,, \,
\limsup_{n\to\infty} \frac1n \log \|Df^n\mid_{E^c}(x)\| \in [\log \lambda, \log \tilde \beta] \cup\{\log\beta\}
$$
and also that there exists a unique point $Q$ (indeed it is a fixed point by $f$) so that the central Lyapunov exponent
is $\log\beta>0$.
Let us consider the H\"older continuous potential $\phi_t=-t\log \|Df\mid_{E^c}\|$ for a large value of negative $t$
and the observable $\psi=\log \|Df\mid_{E^c}\|$.
It follows from \cite[Proposition~5.6]{DG12} that for all $t\ll 0$ the Dirac measure $\delta_Q$ is the unique
equilibrium state for $f$ with respect to $\phi_t$ and consequently
$
\Ptop(f,\phi_t)  = -t \log \|Df(Q)\mid_{E^c}\| = -t \log \beta.
$
\vspace{-.5cm}
\begin{figure}[h]
\includegraphics[scale=.4]{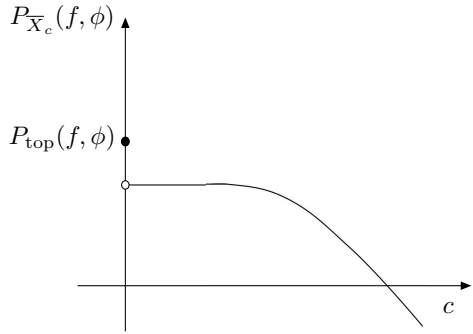}
\vspace{-.7cm}
\caption{Discontinuity of the pressure function}\label{Figure2}
\end{figure}
%
On the other hand if $c\neq \log\beta$ then, taking all invariant measures with central Lyapunov exponent equal to $c$, it follows that
$\sup\Big\{  h_\eta(f) + \int -t \log \|Df(x)\mid_{E^c}\| \,d\eta  : \lambda(\eta)=c \Big\}
     \leq h_{\text{top}}(f) -t\log \tilde \beta$
is strictly smaller than $\Ptop(f,\phi_t)$. This shows the discontinuity of the pressure function
$
c\mapsto P_{\overline{X}_c} (f,\phi_t)
$
where $\overline{X}_c$ is associated to the observable $\psi$.
Actually the same argument leads to prove that
\begin{align*}
\sup\Big\{  h_\eta(f) & + \int -t \log \|Df(x)\mid_{E^c}\| \,d\eta  : \lambda(\eta)\in [\log\lambda, \log \tilde \beta] \Big\} \\
            & = \sup\Big\{  h_\eta(f)  + \int -t \log \|Df(x)\mid_{E^c}\| \,d\eta  :
                    \lambda(\eta)\in [\log\lambda, \log \beta ) \Big\} \\
    & < \Ptop(f,\phi_t)
\end{align*}
and so there exists an interval of constancy for this pressure function.

\subsection*{Acknowledgements.}
This work was partially suppported by CNPq and is
part of the first author's PhD thesis at Federal University of Bahia. The authors are deeply
grateful to V. Climenhaga, D. Kwietniak and M. Todd for useful coments and to the anonymous
referee for many suggestions that helped to improve the manuscript.

\bibliographystyle{alpha}

\begin{thebibliography}{3}

\bibitem{ALFV}
J.~F. Alves, S.~Luzzatto, J.~Freitas, and S.~Vaienti.
\newblock From rates of mixing to recurrence times via large deviations.
\newblock {\em Advances Math.}, 228 (2011) 1203--1236.

\bibitem{BPS97}
L. Barreira, Y. Pesin, and J. Schmeling.
\newblock Multifractal spectra and multifractal rigidity for horseshoes.
\newblock J. Dynam. Control Systems 3 (1997), no. 1, 33--49.

\bibitem{Bomfim}
T. Bomfim, A. Castro and P. Varandas.
\newblock Differentiability of thermodynamical quantities in non-uniformly expanding dynamics,
\newblock {\em Preprint ArXiv:1205.5361}.

\bibitem{BV14}
T. Bomfim and P. Varandas.
\newblock Multifractal analysis of the irregular set for almost-additive sequences via large deviations
\newblock Preprint ArXiv:1410.2220.


\bibitem{Bowen}
R. Bowen.
\newblock Some systems with unique equilibrium states.
\newblock {\em Math. Systems Theory} 8, 193 - 202,
1974.

\bibitem{Bow75}
R. Bowen.
\newblock Equilibrium States and the Ergodic Theory of Anosov Diffeomorphisms.
\newblock {\em Lect. Notes in Math.}, 470, Springer, 1975.

\bibitem{BR75}
R. Bowen and D. Ruelle.
\newblock The ergodic theory of Axiom A flows.
\newblock  {\em Invent. Math.}, 29, no. 3, 181--202, 1975.



\bibitem{BT09}
H.~Bruin and M. Todd.
\newblock Equilibrium states for interval maps: the potential $-t log|Df|$.
\newblock {\em Ann. Sci. ƒcole Norm. Sup.},  4: 42, 559--600, 2009.

\bibitem{CV13}
A.~Castro, P.~Varandas.
\newblock Equilibrium states for non-uniformly expanding
maps: decay of correlations and strong stability.
\newblock {\em Annales de l Institut Henri Poincar\'e - Analyse non Lineaire},
30:2, 225--249, 2013.

\bibitem{CN14}
A.~Castro and T.~Nascimento.
\newblock Statistical properties of equilibrium states for partially hyperbolic attractors.
\newblock {\em Preprint UFBA}, 2014.

\bibitem{Climenhaga}
V.~Climenhaga
\newblock Multifractal formalism derived from thermodynamics for general dynamical systems.
\newblock {\em Electronic Research Announcements in Mathematical Sciences}, 17, 1-11, 2010.

\bibitem{Cli13}
V.~Climenhaga
\newblock Topological pressure of simultaneous level sets.
\newblock {\em Nonlinearity} 26 (2013), 241--268.

\bibitem{CTY13}
V.~Climenhaga, D.~Thompson and K. Yamamoto.
\newblock Large deviations for systems with non-uniform structure.
\newblock {\em Preprint ArXiv:1304.5497} .

\bibitem{CRL98}
H.~Comman and J.~Rivera-Letelier.
\newblock Large deviations principles for non-uniformly hyperbolic rational
  maps.
\newblock {\em Ann. Inst. H. Poincar\'e Anal. Non Lin\'eaire}, 15:539--579,
  1998.

\bibitem{DK}
M.~Denker and M.~Kesseb\"{o}hmer.
\newblock Thermodynamical formalism, large deviation and multifractals.
\newblock {\em Stochastic Climate Models, Progress in Probability} 49, 159-169, 2001.

\bibitem{DG12}
L. D\'iaz and K. Gelfert.
\newblock Porcupine-like horseshoes: transitivity, Lyapunov spectrum, and phase transitions.
\newblock \emph{Fund. Math.} 216 (2012), no. 1, 55--100.

\bibitem{DGR11}
L. D\'iaz, K. Gelfert and M. Rams.
\newblock Rich phase transitions in step skew products.
\newblock \emph{Nonlinearity}  24 (2011), no. 12, 3391--3412.

\bibitem{EKW}  A. Eizenberg, Y. Kifer and B. Weiss,
\newblock  Large deviations for $\mathbb{Z}^d$ actions,
\newblock  \emph{Comm. Math. Phys.} ,  { 164} (1994), 433--454.

\bibitem{Feng}
D. Feng and W. Huang.
\newblock Lyapunov spectrum of asymptotically sub-additive potentials.
\newblock {\em Commun.
Math. Phys.}, 297: 1-43, 2010.

\bibitem{GR}
K. Gelfert and M. Rams.
\newblock The Lyapunov spectrum of some parabolic systems.
\newblock {\em Ergod. Th. Dynam. Sys.}, 29 (3), 919-940, 2009.

\bibitem{Hof77}
F. Hofbauer.
\newblock Examples for the nonuniqueness of the equilibrium state
\newblock {\em Trans. Amer. Math. Soc.}, 228, 223--241, 1977.

\bibitem{JR} T. Jordan and M. Rams.
\newblock Multifractal analysis of weak Gibbs measures for non-uniformly
expanding $C^1$ maps.
\newblock {\em Ergod. Th. Dynam. Sys.}, 31 (1), 143-164, 2011.

\bibitem{special}
V. Kleptsyn, D. Ryzhov and S. Minkov.
\newblock Special ergodic theorems and  dynamical large deviations.
\newblock {\em Nonlinearity} 25, 3189-3196, 2012.

\bibitem{Mel09}
I.~Melbourne.
\newblock Large and moderate deviations for slowly mixing dynamical systems.
\newblock {\em Proc. Amer. Math. Soc.}, 137:1735--1741, 2009.

\bibitem{MN08}
I.~Melbourne and M.~Nicol.
\newblock Large deviations for nonuniformly hyperbolic systems.
\newblock {\em Trans. Amer. Math. Soc.}, 360:6661--6676, 2008.

\bibitem{pesin}
Ya. Pesin.
\newblock Dimension theory in dynamical systems.
\newblock {\em University of Chicago Press}, Contemporary views and applications, 1997.

\bibitem{PW97}
Y. Pesin and H. Weiss.
\newblock The multifractal analysis of Gibbs measures: Motivation, mathematical foundation, and examples.
\newblock {\em Chaos}, 7(1):89--106, 1997.

\bibitem{PS05}
C. Pfister and W. Sullivan.
\newblock Large deviations estimates for dynamical systems without the specification
property. Applications to the $\beta$-shifts.
\newblock {\em Nonlinearity}, 18 (2005) 237-261.

\bibitem{PS09}
M.~Pollicott and R.~Sharp.
\newblock Large deviations for intermittent maps.
\newblock {\em Nonlinearity}, 22, 2079--2092 (2009).

\bibitem{PW99}
M. Pollicott and H. Weiss.
\newblock Multifractal Analysis of Lyapunov Exponent for
Continued Fraction and Manneville-Pomeau Transformations and Applications to Diophantine
Approximation.
\newblock {\em Commun. Math. Phys.}, 207, 145 --171 (1999).

\bibitem{RY08}
L.~Rey-Bellet and L.-S. Young.
\newblock Large deviations in non-uniformly hyperbolic dynamical systems.
\newblock {\em Ergod. Th. Dynam. Sys.}, 28: 587-612, 2008.

\bibitem{TV99}
F. Takens and E. Verbitski.
\newblock Multifractal analysis of local entropies for
expansive homeomorphisms with specification.
\newblock {\em Comm. Math. Phys.}, 203:593--612, 1999.

\bibitem{T}
M. Todd.
\newblock Multifractal analysis for multimodal maps.
\newblock {\em Preprint Arxiv}, 2008.

\bibitem{top2}
D. Thompson.
\newblock A variational principle for topological pressure for certain non-compact sets.
\newblock {\em J. London Math. Soc.}, 80 (3) : 585-- 602, 2009.

\bibitem{Daniel}
D. Thompson.
\newblock The irregular set for maps with the specification property has full topological pressure.
\newblock {\em Dyn. Syst.}, v. 25, no. 1, p. 25--51, 2010.

\bibitem{Varandas2}
P. Varandas and M. Viana.
\newblock Existence, uniqueness and stability of equilibrium states for
non-uniformly expanding maps.
\newblock {\em Annales de l'Institut Henri Poincar\'e- Analyse Non-Linéaire}, 27, 555-593, 2010.

\bibitem{Va12}
P.~Varandas.
\newblock Non-uniform specification and large deviations for weak Gibbs measures.
\newblock {\em Journal of Stat. Phys.},  146, 330--358, 2012.

\bibitem{Va13}
P. Varandas and Y. Zhao.
\newblock Weak specification properties and large deviations for non-additive potentials.
\newblock {\em Ergodic Theory and Dynamical Systems (Print)}, 35:3 p. 968--993, 2015.

\bibitem{Young}
L.-S. Young.
\newblock Some large deviations for dynamical systems.
\newblock {\em Trans. Amer. Math. Soc.}, 318: 525-543, 1990.


\bibitem{ZC13}
X. Zhou and E. Chen.
\newblock Multifractal analysis for the historic set in topological dynamical systems.
\newblock {\em Nonlinearity}, 26, no. 7, 1975--1997, 2013.

\end{thebibliography}

\end{document}